\documentclass[12pt, leqno]{amsart}

\usepackage{graphicx}
\usepackage{amssymb,amsmath,amsthm,amscd,accents}
\usepackage{mathrsfs}
\usepackage{lscape}
\usepackage{enumerate}
\usepackage{upgreek}
\usepackage[usenames,dvipsnames]{color}
\usepackage[colorlinks=true, pdfstartview=FitV,
 linkcolor=blue,citecolor=blue,urlcolor=blue]{hyperref}
\usepackage{tikz}
\usetikzlibrary{arrows.meta,arrows}
\usepackage{bm}
\usepackage{marginnote}
\usepackage{verbatim}
\usepackage[normalem]{ulem}  

\usepackage[all]{xy}

\allowdisplaybreaks[3]

\newcommand{\nc}{\newcommand}
\numberwithin{equation}{section}

\newenvironment{red}{\relax\color{red}}{\relax}
\newenvironment{blue}{\relax\color{blue}}{\hspace*{.5ex}\relax}
\newenvironment{purple}{\relax\color{blue}}{\hspace*{.5ex}\relax}
\newenvironment{magenta}{\relax\color{magenta}}{\hspace*{.5ex}\relax}
\newenvironment{yellow}{\relax\color{yellow}}{\hspace*{.5ex}\relax}

\newcommand{\bep}{\begin{purple}}
\newcommand{\eep}{\end{purple}}

\newcommand{\ber}{\begin{red}}
\newcommand{\er}{\end{red}}
\newcommand{\beb}{\begin{blue}}
\newcommand{\eb}{\end{blue}}
\newcommand{\bem}{\begin{magenta}}
\newcommand{\eem}{\end{magenta}}
\newcommand{\bey}{\begin{yellow}}
\newcommand{\eey}{\end{yellow}}

\newcommand{\bero}[1]{\begin{red}{}\marginnote{\fbox{\scshape\lowercase{O}}}%
#1}

\newcommand{\bebo}[1]{\begin{blue}{}\marginnote{\fbox{\scshape\lowercase{O}}}%
#1}

\newcommand{\berE}[1]{\begin{red}{}\marginnote{\fbox{\scshape\lowercase{E}}}%
#1}

\nc{\hs}{\hspace*}
\nc{\ms}{\mspace}

\nc{\qR}[1]{\ttq_{\mspace{-2mu}\raisebox{-.8ex}{${\scriptstyle{#1}}$}}}

\theoremstyle{plain}
\newtheorem{lemma}{Lemma}[section]
\newtheorem{proposition}[lemma]{Proposition}
\newtheorem{theorem}[lemma]{Theorem}

\newtheorem{corollary}[lemma]{Corollary}

\theoremstyle{definition}
\newtheorem{remark}[lemma]{Remark}
\newtheorem{example}[lemma]{Example}

\newtheorem{definition}[lemma]{Definition}

\newtheorem{assumption}[lemma]{Assumption}

\renewcommand{\le}{\leqslant}
\renewcommand{\ge}{\geqslant}

\newcommand{\Opd}{\mathop{\prod}}

\newcommand{\seteq}{\mathbin{:=}}

\newcommand{\ev}{{\operatorname{ev}}}
\newcommand{\tens}{\mathop\otimes}

\newcommand{\g}{\mathfrak{g}}

\newcommand{\n}{\mathfrak{n}}

\newcommand{\Q}{\mathbb{Q}}
\newcommand{\Z}{\mathbb{Z}\ms{1mu}}

\newcommand{\al}{{\ms{1mu}\alpha}}

\newcommand{\la}{\lambda}
\newcommand{\be}{{\ms{1mu}\beta}}

\newcommand{\wt}{{\rm wt}}

\newcommand{\Seq}{{\rm Seq}}

\newcommand{\het}{{\rm ht}}

\newcommand{\llf}{(\hspace{-0.7ex}(}
\newcommand{\LLf}{\left(\hspace{-1ex}\left(}
\newcommand{\rrf}{)\hspace{-0.7ex})}
\newcommand{\RRf}{\right)\hspace{-1ex}\right)}

\newcommand{\pair}[1]{  \llf #1 \rrf  }
\newcommand{\Bpair}[1]{  \LLf #1 \RRf  }
\newcommand{\ii}{ \textbf{\textit{i}}}

\newcommand{\jj}{ \textbf{\textit{j}}}

\newcommand{\End}{\operatorname{End}}

\newcommand{\hI}{\widehat{I}}

\newcommand{\hcalA}{\widehat{\calA}}

\newcommand{\sfC}{\mathsf{C}}

\newcommand{\sfF}{\mathsf{F}}
\newcommand{\sfe}{\mathsf{e}}

\newcommand{\sfD}{\mathsf{D}}
\newcommand{\sfP}{\mathsf{P}}
\newcommand{\sfR}{\mathsf{R}}
\newcommand{\sfQ}{\mathsf{Q}}
\newcommand{\sfW}{\mathsf{W}}
\newcommand{\sff}{\mathsf{f}}

\newcommand{\sfd}{\mathsf{d}}

\newcommand{\bbA}{\mathbb{A}}

\newcommand{\bfk}{\mathbf{k}}

\newcommand{\bfZ}{\mathbf{Z}}

\newcommand{\bfg}{\mathbf{g}}

\newcommand{\bfd}{\mathbf{d}}

\newcommand{\calU}{\mathcal{U}}
\newcommand{\calA}{\mathcal{A}}

\newcommand{\calF}{\mathcal{F}}
\newcommand{\calB}{\mathcal{B}}
\newcommand{\calK}{\mathcal{K}}
\newcommand{\calP}{\mathcal{P}}

\newcommand{\scrC}{\mathscr{C}}

\newcommand{\ttP}{\mathtt{P}}
\newcommand{\ttB}{\mathtt{B}}

\newcommand{\ttb}{\mathtt{b}}

\newcommand{\ttx}{\mathtt{x}}
\newcommand{\tty}{\mathtt{y}}
\newcommand{\ttz}{\mathtt{z}}
\newcommand{\ttd}{\mathtt{d}}

\newcommand{\ttq}{{\ms{1mu}\mathtt{q}\ms{1mu}}}

\newcommand{\ttJ}{\mathtt{J}}

\newlength{\mylength}
\newcommand*{\para}{%
  \rlap{\rotatebox{-30}{\rule[.05ex]{.4pt}{.77em}}}%
  \kern.04em%
  \rlap{\kern.36em\raisebox{0.649519052835em}{\rule{.6em}{.4pt}}}%
  \rule{.6em}{.4pt}\kern-.04em%
  \rotatebox{-30}{\rule[.05ex]{.4pt}{.77em}}}

\newcommand{\diag}{\mathrm{diag}}
\newcommand{\rmQ}{\mathrm{Q}}

\newcommand{\wl}{\mathsf{P}}
\newcommand{\rl}{\sfQ}

\newcommand{\weyl}{\sfW}
\newcommand{\lan}{\langle}
\newcommand{\ran}{\rangle}

\newcommand{\isoto}[1][]{\mathop{\xrightarrow%
[{\raisebox{.3ex}[0ex][.3ex]{$\scriptstyle{#1}$}}]%
{{\raisebox{-.6ex}[0ex][-.6ex]{$\mspace{2mu}\sim\mspace{2mu}$}}}}}

\newcommand{\ee}{\end{enumerate}}
\newcommand{\bitem}{\begin{itemize}}
\newcommand{\eitem}{\end{itemize}}
\newcommand{\ben}{\begin{enumerate}[{\rm (1)}]}
\newcommand{\bnum}{\begin{enumerate}[{\rm (i)}]}
\newcommand{\bnump}{\begin{enumerate}[{\rm (i)$'$}]}
\newcommand{\bna}{\begin{enumerate}[{\rm (a)}]}
\newcommand{\bnA}{\begin{enumerate}[{\rm (A)}]}
\newcommand{\bc}{\begin{cases}}
\newcommand{\ec}{\end{cases}}

\newenvironment{myequation}
{\relax\setlength{\arraycolsep}{1pt}\begin{eqnarray}}
{\end{eqnarray}}
\newenvironment{myequationn}
{\relax\setlength{\arraycolsep}{1pt}\begin{eqnarray*}}
{\end{eqnarray*}}

\nc{\eq}{\begin{myequation}}
\nc{\eneq}{\end{myequation}}
\nc{\eqn}{\begin{myequationn}}
\nc{\eneqn}{\end{myequationn}}
\nc{\cl}{\colon}
\nc{\ake}[1][1ex]{\rule[-#1]{0ex}{1ex}}
\nc{\akew}[1][1ex]{\rule[-1ex]{#1}{0ex}}
\nc{\akeu}[1][1ex]{\rule[#1]{0ex}{1ex}}
\nc{\id}{\mathrm{id}}
\nc{\bl}{\bigl(}
\nc{\br}{\bigr)}
\nc{\qt}[1]{\quad\text{#1}}
\nc{\qtq}[1][{and}]{\quad\text{{#1}}\quad}
\nc{\eqs}[1]{\underset{\raisebox{.4ex}[.7ex][0ex]{$\scriptstyle{#1}$}}{=}}
\nc{\snoi}{\smallskip \noindent}
\nc{\mnoi}{\medskip \noindent}
\nc{\Mat}{\mathrm{Mat}}
\nc{\ol}{\overline}
\nc{\ul}{\underline}
\nc{\ang}[1]{\boldsymbol{\langle}{#1}\boldsymbol{\rangle}}
\nc{\rang}[1]{\boldsymbol{\langle}{#1}\boldsymbol{\rangle} \hspace{-.6ex} \boldsymbol{\rangle}}
\nc{\ba}{\begin{array}}
\nc{\ea}{\end{array}}
\nc{\noi}{\noindent}
\nc{\evq}{ \ev_{q=1}}
\nc{\Um}{ \calU^-_q(\g)}

\nc{\yi}{x_{i}}
\nc{\yj}{x_{j}}
\nc{\yk}{x_{k}}

\nc{\yjm}{x_{j,m}}
\nc{\yjmp}{x_{j,m+1}}
\nc{\yjmm}{x_{j,m-1}}
\nc{\yim}{x_{i,m}}
\nc{\yip}{x_{i,p}}
\nc{\yipp}{x_{i,p+1}}
\nc{\yipm}{x_{i,p-1}}
\nc{\yjp}{x_{j,p}}
\nc{\yjpp}{x_{j,p+1}}
\nc{\yjpm}{x_{j,p-1}}
\nc{\yimp}{x_{i,m+1}}
\nc{\yimm}{x_{i,m-1}}
\nc{\ykm}{x_{k,m}}
\nc{\ynm}{x_{n,m}}
\nc{\ynmp}{x_{n,m+1}}
\nc{\ynmpp}{x_{n,m+2}}
\nc{\ynnm}{x_{n-1,m}}
\nc{\ynnmp}{x_{n-1,m+1}}
\nc{\ynnmpp}{x_{n-1,m+2}}
\nc{\ykmp}{x_{k,m+1}}
\nc{\ykmm}{x_{k,m-1}}
\nc{\ykp}{x_{k,p}}
\nc{\ykpp}{x_{k,p+1}}
\nc{\ykpm}{x_{k,p-1}}

\nc{\seq}[1]{ \boldsymbol{(} {#1} \boldsymbol{)}   }

\newcommand{\bse}{{\boldsymbol{e}}}

\newcommand{\bsd}{{\boldsymbol{d}}}

\newcommand{\bsa}{{\boldsymbol{a}}}
\newcommand{\bsc}{{\boldsymbol{c}}}
\newcommand{\bsu}{{\boldsymbol{u}}}
\newcommand{\bsv}{{\boldsymbol{v}}}
\newcommand{\bst}{{\boldsymbol{t}}}

\nc{\rvt}[2]{ {\be^{#1}_{#2}} }
\nc{\pvt}[2]{ {f^{#1}_{#2}} }
\nc{\qbA}[1]{ {\hcalA_{#1}} }
\nc{\qbAu}[1]{ {\hcalA^{#1}} }
 
\usepackage[colorinlistoftodos]{todonotes}

\nc{\cmC}{\mathsf{C}}
\nc{\bR}{\mathbf{k}}
\nc{\qA}{\widehat{\mathcal{A}}_\g}
\nc{\bT}{\mathcal{T}}
\nc{\bS}{\mathrm{S}}
\nc{\cC}{\kappa}
\nc{\Ht}{\mathrm{ht}}
\nc{\rlQ}{\mathsf{Q}}
\nc{\Span}{\mathrm{Span}}
\nc{\ri}{r_i}
\nc{\ir}{{_ir}}
\nc{\bfi}{\mathbf{i}}
\nc{\bfj}{\mathbf{j}}
\nc{\bft}{\mathbf{t}}
\nc{\bfu}{\mathbf{u}}
\nc{\bfv}{\mathbf{v}}
\nc{\rF}{\mathsf{F}}
\nc{\rG}{\mathsf{X}}
\nc{\rH}{\mathsf{Y}}
\nc{\rFv}{\mathsf{G}}
\nc{\rFt}{\widetilde\rF}
\nc{\dD}{\mathcal{D}}
\nc{\inv}{\zeta}
\nc{\bg}{\ttB}
\nc{\gar}{\Updelta}
\nc{\Ep}{E'}
\nc{\Es}{E^{\star}}
\nc{\es}{e^\star}
\nc{\bbe}{\mathsf{e}}
\nc{\bbf}{\mathsf{f}}
\newcommand{\abs}[1]{ \widetilde {#1} }
\nc{\bfic}{\bfi_{\circ}}
\nc{\bficp}{\bfi_{\circ}^\prime}
\nc{\pbwB}{\mathcal{B}}
\nc{\pd}{\mathfrak{I}}
\nc{\pdh}{\mathfrak{H}}
\nc{\gf}{{\g_{\mathrm{fin}}}}
\nc{\catC}{\scrC}
\nc{\hg}{\mathfrak{G}}

\title[PBW theory for Bosonic extensions of quantum groups]{PBW theory for Bosonic extensions of quantum groups}

\author[S.-j.~Oh]{Se-jin Oh } 
\address[S.-j.~Oh]{ Department of Mathematics, Sungkyunkwan University, Suwon, South Korea}
\email{sejin092@gmail.com}
\urladdr{https://sites.google.com/site/mathsejinoh/}
\thanks{ S.-j.\ Oh was supported by the Ministry of Education of the Republic of Korea and the National Research Foundation of Korea (NRF-2022R1A2C1004045).}

\author[E. Park]{Euiyong Park}
\thanks{The research of E.\ Park was supported by the  National Research Foundation of Korea(NRF) grant funded by the Korea government (MSIT) (RS-2023-00273425 and NRF-2020R1A5A1016126).}
\address[E. Park]{Department of Mathematics, University of Seoul, Seoul 02504, Korea}
\email[E. Park]{epark@uos.ac.kr}
 
\date{February 7, 2024}

\begin{document}

\begin{abstract}
In this paper, we develop the PBW theory for the bosonic extension $\qbA{\g}$ of a quantum group $\mathcal{U}_q(\g)$ of \emph{any} finite type. 
When $\g$ belongs to the class of \emph{simply-laced type}, the algebra $\qbA{\g}$ arises from the quantum Grothendieck ring of the Hernandez-Leclerc category over 
quantum affine algebras of untwisted affine types. 
We introduce and investigate a symmetric bilinear form $\pair{\ , \ }$ on $\qbA{\g}$ which is invariant under the braid group actions $\bT_i$ on $\qbA{\g}$, and 
study the adjoint operators $\Ep_{i,p}$ and $\Es_{i,p}$ with respect to $\pair{\ , \ }$.    
It turns out that the adjoint operators $\Ep_{i,p}$ and $\Es_{i,p}$ are analogues of the $q$-derivations $e_i'$ and $\es_i$ on the negative half $\calU_q^-(\g)$ of $\calU_q(\g)$. 
Following this, we introduce a new family of subalgebras denoted as $\qbA{\mathfrak{g}}(\ttb)$ in $\qbA{\mathfrak{g}}$. These subalgebras are defined for any elements $\ttb$ in the positive submonoid $\bg^+$ of the 
(generalized) braid group $\ttB$ of $\g$. We prove that $\qbA{\mathfrak{g}}(\ttb)$ exhibits PBW root vectors and PBW bases defined by $\bT_\ii$ for any sequence $\ii$ of $\ttb$.
The PBW root vectors satisfy a Levendorskii-Soibelman formula and the PBW bases are orthogonal with respect to $\pair{\ , \ }$. The algebras $\qbA{\g} (\ttb)$ can be understood as a natural extension of quantum unipotent coordinate rings.  
\end{abstract}

\setcounter{tocdepth}{1}

\maketitle
\tableofcontents

\section{Introduction}

Let $q$ be an indeterminate and let $\hg$ be an affine Kac-Moody Lie algebra. 
The category $\catC_{\hg}$ of finite-dimensional integrable modules over the quantum affine algebra $\calU_q'(\hg)$ takes a central and important part in the study of representations of quantum affine algebras (see~\cite{CH10,HL21} and references therein).
 The \emph{Hernandez-Leclerc category} $\catC_{\hg}^0$ is a distinguished monoidal subcategory of $\catC_{\hg}$ defined by certain fundamental modules (see \cite{HL10} for precise definition). Hernandez-Leclerc showed a cluster-algebraic approach to the Grothendieck ring $K(\catC_{\hg}^0)$ of the Hernandez-Leclerc category $\catC_{\hg}^0$, which led us to several remarkable research results (see \cite{FHOO, FHOO2, HL10, HL15, HL16} for examples).
The \emph{quantum Grothendieck ring}  $\calK_t(\catC_{\hg}^0)$ is a $t$-deformation of the Grothendieck ring $K(\catC_{\hg}^0)$ inside a quantum torus, where $t$ is an indeterminate (\cite{Her04, Nak04,VV03}).  
It turns out that there exist several subalgebras of $\calK_t(\catC_{\hg}^0)$ having (quantum) cluster algebra structures  of \emph{skew-symmetric type}  
with  initial  clusters  arising from the \emph{$(q,t)$-characters} of Kirillov–Reshetikhin modules (see \cite{FHOO2, HL10, HL16}).
For the sake of simplicity, we assume that $\hg$ is of simply-laced affine ADE type, and write  $\bfg$ for the finite ADE simple Lie algebra inside $\hg$ with an index set $I$.  
As an algebra, a presentation of  $ \calK_t(\catC_{\hg}^0)$  was studied in \cite{HL15}. The algebra $\calK_t(\catC_{\hg}^0)$ has an infinitely many generators $f_{i,p}$ ($i\in I$, $p\in \Z$) and they satisfy the quantum Serre relations  among $\{ f_{i,p} \}_{i \in I}$   and $t$-boson relations  among $\{ f_{i,p}, f_{j,p+1} \}_{i,j \in I}$ for each $p \in \Z$. 
Thus, for each $p\in \Z$, the subalgebra of $\calK_t(\catC_{\hg}^0)$ generated by $f_{i,p}$ ($i\in I$) is isomorphic to the quantum group $\calU_t^-(\bfg)$. This observation suggests that the algebra $\calK_t(\catC_{\hg}^0)$ can be regarded as an \emph{affinization} of $\calU_t^-(\bfg)$.
We denote by $\qbA{\bfg}$ the algebra defined formally by the generators $f_{i,p}$ ($i\in I$, $p\in \Z$) and the same defining relations as in $\calK_t(\catC_{\hg}^0)$, which is called the \emph{bosonic extension} of the quantum group $\calU_t(\bfg)$. By construction, there is an isomorphism $\qbA{\bfg} \simeq  \calK_t(\catC_{\hg}^0)$. 
Despite the significance of the algebra $\qbA{\bfg}$, the algebra does not seem to have been intensively explored from the perspective of ring theory (we refer \cite[Introduction \S 1.3]{HL15}).
Our interest lies in the algebra structure of $\qbA{\g}$  for \emph{any} simple Lie algebra $\g$.   

The bosonic extension $\qbA{\g}$ is defined in Definition \ref{Def: extended qg} for a Cartan matrix $\cmC = (c_{i,j})_{i,j\in I}$ of \emph{arbitrary finite type}. In Definition \ref{Def: extended qg}, we use $q$ as an indeterminate to follow the usual notation in quantum groups. Remark \ref{Rem: q and t} explains the relation between two indeterminates $q$ and $t$. 
Note that, for any quantum affine algebra  $\calU_q'(\hg)$ of \emph{arbitrary  untwisted affine type},  
the quantum Grothendieck ring $\calK_t(\catC_{\hg}^0)$ is isomorphic to the bosonic extension $\qbA{\bfg}$ associated with some \emph{symmetric} Cartan matrix $\sfC_\bfg$ (see Remark \ref{Rem: q and t}). When Cartan matrix $\cmC$ is of \emph{non-symmetric type}, the bosonic extension $\qbA{\g}$ is isomorphic to the \emph{quantum virtual Grothendieck ring} introduced in 
\cite{KO22,JLO2, JLO1}.  Similar to quantum Grothendieck ring, the quantum virtual Grothendieck ring is realized as a subalgebra inside a quantum torus and possesses a (quantum) cluster algebra structure. 
However, the asymmetry of $\qbA{\g}$ aligns with that of $\g$, whereas the asymmetry of $\calK_t(\scrC_\hg^0)$ is consistently skew-symmetric (see \cite{JLO1}).
It was shown that $\qbA{\g}$ has an action of the (generalized) braid group $\bg_\cmC$ associated with the Cartan matrix $\cmC$ (see \cite{KKOP21A} for symmetric types and \cite{JLO2} for general types). The braid group action $\bT_i$ ($i\in I$) on $\qbA{\g}$ is similar to Lusztig's braid symmetry on $\calU_q(\g)$ (\cite{LusztigBook}) in local parts and is understood as a ring-theoretic counterpart of the braid group action on the \emph{extended crystal} (see \cite{KP22, Park23}). It is conjectured in \cite{KKOP21A} that the braid group actions $\bT_i$ can be lifted to autofunctors on the category  $\catC_\hg^0$.  

In this paper, we introduce a family of new subalgebras $\qbA{\g} (\ttb)$ of $\qbA{\g}$ for any elements $\ttb$ in the positive submonoid $\bg_\cmC^+$ of the braid group $\bg_\cmC$ associated with $\cmC$
and develop the \emph{Poincar\'e-Birkhoff-Witt theory} (shortly PBW theory) for the subalgebra $\qbA{\g} (\ttb)$ using the braid group actions $\bT_i$ ($i\in I$). 
Let $\weyl_\cmC$ be the Weyl group associated with $\cmC$ and $\pi: \bg_\cmC^+ \twoheadrightarrow \weyl_\cmC$ the natural projection (see \eqref{Eq: b->W}).  
When there is $w\in \weyl_\cmC$ such that $w = \pi(\ttb)$ and $\ell(w) = \ell(\ttb)$, 
the subalgebra $\qbA{\g} (\ttb)$ coincides with the \emph{quantum unipotent coordinate algebra} $A_q(\n(w))$. 
Note that the algebra $A_q(\n(w))$ is defined by the (dual) PBW basis using Lusztig's braid symmetries and 
takes a crucial role in the intersection of quantum groups and cluster algebras (see \cite{GLS11,GLS13, GY17, KKKO18, Kimura12} for examples).  

Since the algebra $\qbA{\g}$ serves as an affinization of $\calU_q^-(\g)$, one can regard the algebra $\qbA{\g} (\ttb)$ as a natural generalization of $A_q(\n(w))$,   
and the PBW theory developed in the paper is also understood as an extension of the PBW theory for $A_q(\n(w))$ through the braid group $\ttB$.

Let us explain our results more precisely. We write $\qbA{}$ for $\qbA{\g}$ for simplicity and, for an interval $\ttJ$, define $\qbA{\ttJ}$ to be the subalgebra of $\qbA{}$ generated by $f_{i,p}$ for $i\in I $ and $p\in \ttJ$. Note that $ \qbA{[k]} \simeq \calU_q^-(\g)$ for any single interval $[k]:=\{ k \} \subset \Z$. 
In Definition \ref{def: ( )}, we define a bilinear form 
$\pair{ \ , \  }$ on $\qbA{[a,b]}$, which can be understood as a braid-analogue of the bilinear forms 
$(\cdot,\cdot)_L$ of Lusztig and $(\cdot,\cdot)_K$ of Kashiwara on $\calU_q^-(\g)$
in a simultaneously way, through the serial decomposition of $\qbA{[a,b]}$: 
$$
\qbA{[a,b]} \simeq \qbA{[b]} \otimes_{\bR} \qbA{[b-1]} \otimes_{\bR} \cdots \otimes_{\bR}  \qbA{[a]}
$$
(see Lemma \ref{Lem: qA decomp}).  
 Lemma \ref{Lem: properties of (,)} asserts that the bilinear form $\pair{\ , \ }$ is both symmetric and non-degenerate. Moreover, it exhibits compatibility with (anti-)automorphisms $\dD$ and $*$.   
Similar to the quantum group case, it is revealed that the bilinear form $\pair{\ , \ }$ remains invariant under the braid group action $\bT_i$ for $i\in I$ (see Proposition \ref{Prop: invariant form}).

For a \emph{PBW datum} $\pd = (\ii, \xi)$, which is a pair of $ \ii = (i_u, i_{u+1}, \ldots, i_v) \in I^{[u,v]}$ and $\xi\in \Z$, we define the \emph{PBW root vectors}
$$
\rF^{\pd}_{k} \seteq \bT_{i_u} \bT_{i_{u+1}} \cdots \bT_{i_{k-1}} (f_{i_k,{\xi}}) \qquad \text{for any $k\in [u,v]$},
$$
and the \emph{PBW monomials}
$$
\rF_{\pd}(\bsd) \seteq \Opd^{\longrightarrow}_{k\in [u,v]}  \rF^{d_k}_k  =    \rF^{d_v}_v    \rF^{d_{v-1}}_{v-1}   
\cdots  \rF^{d_u}_u \qquad \text{for any $\bsd = (d_u, d_{u+1}, \ldots d_v) \in \Z_{\ge 0}^{\oplus [u,v]}$}.
$$
The $*$-twisted PBW root vectors $\rFv^{ \pd}_{k}$ and PBW monomials $\rFv_{\pd}(\bsd)$ are also defined (see \eqref{Eq: PBW vectors} and \eqref{Eq: PBW monomials}) and the relation between those two PBW root vectors is investigated in Lemma \ref{lem: F anf vF}.

We initially consider with the case when $\bfi \in I^{[1, N]}$ is a \emph{locally reduced sequence} (see \eqref{Eq: longest case}).
Here $N = \ell \times(b-a+1)$ represents a multiple of the length $\ell = \ell (w_\circ)$ of the longest element $w_\circ \in \weyl_\cmC$. Lemma \ref{Lem: pbw comm}  asserts that the PBW root vectors $\rF^{\pd}_{k}$ satisfy an analogue of \emph{Levendorskii-Soibelman formula} (\cite{LS91}) (shortly LS-formula).
This formula is a convex property among $\rF^{\pd}_{k}$.  
We establish that the PBW monomials $\rF_{\pd}(\bsd)$ form an orthogonal basis for $\qbA{[a,b]}$ (encompassing the entire algebra $\qbA{}$) with respect to the bilinear form $\pair{\ , \ }$. This is achieved by providing an explicit formula for $ \pair{\rF_{\pd}(\bsd), \rF_{\pd}(\bsd')}$ 
(see Theorem \ref{Thm: orthonormal_reduced} and Corollary \ref{Cor: orthonormal_reduced_Aq}). Proposition \ref{Prop: T and S} tells us that the braid group action $\bT_i$ is locally equivalent to Lusztig's braid symmetry $\bS_i$ (up to a non-zero scalar multiple). This equivalence plays a crucial role in the course of our proofs. 
In the case of $\qbA{ } \simeq  \calK_t(\catC_{\hg}^0)$, the fundamental modules and standard modules in the category $\catC_{\hg}^0$ give vectors and  
a basis  of $\qbA{ }$ via $(q,t)$-characters (\cite{FHOO,Her04, Nak04, VV03}), which coincide with the PBW root vectors and PBW monomials associated with a 
sequence $\bfi$ \emph{adapted to a $\rmQ$-datum $Q$} (see \cite{FO21,FHOO}). Since $Q$-adapted sequences are examples of locally reduced 
sequences, our PBW vectors and PBW bases generalize the vectors and basis arising from fundamental modules and standard modules. Note that the PBW 
theory for $\qbA{}$ associated with locally reduced sequences can be understood as a ring-theoretic counterpart of the PBW theory for $\catC_\hg^0$ 
developed in \cite{KKOP23}.

For the full generality, we consider \emph{any} element $\ttb\in \bg^+$. 
Let $\pd = (\ii, \xi)$ be a PBW datum of $\ttb$, i.e,  $ \ii \in \Seq(\ttb)$ (see Definition \ref{def: reduced} for $\Seq(\ttb)$) and $\xi\in \Z$, and let $L = \ell(\ttb)$. 
For any interval $\ttJ \subset [1,L]$, 
we denote by $\qbAu{\ttJ,\pd} $ the subalgebra of $\qbA{}$ generated by the PBW root vectors $\rF^{\pd}_k$ for all $  k \in \ttJ$.
In the case when $\xi = 0$ and $\ttJ=[1,L]$, we simple write $\qbA{}(\ttb)\seteq \qbAu{[1,N],\pd}$. 
Lemma \ref{Lem: qA(pd, rs)} says that any subalgebra $\qbAu{\ttJ,\pd} $ can be reduced to the case of the form $\qbA{}(\ttb')$ for some $\ttb' \in \bg^+$ from a ring theoretic point of view. 
 
Utilizing the \emph{Garside element} $\gar$ in $\bg^+$ and the PBW theory linked to locally reduced sequences, we establish an analogue of the LS-formula for PBW root vectors $\rF^{\pd}_{k}$. Consequently, we demonstrate that the PBW monomials $\rF_{\pd}(\bsd)$ form an orthogonal basis for $\qbA{}(\ttb)$ with respect to the bilinear form $\pair{\ , \ }$ (refer to Theorem \ref{Thm: main}).
A crucial element in our proof involves the use of \emph{adjoint operators} $\Ep_{i,p}$ and $\Es_{i,p}$ introduced in Section \ref{Sec: adjoint operators} (see Definition \ref{def: Adjoint operator}). These operators, $\Ep_{i,p}$ and $\Es_{i,p}$, serve as adjoints to the left and right multiplication of $f_{i,p}$ with respect to $\pair{\ , \ }$ respectively.
We give an explicit formulas for applying $\Ep_{i,p}$ and $\Es_{i,p}$ to homogeneous elements (Theorem \ref{Thm: E_ik, f_ik}) and show that $\Ep_{i,p}$ and $\Es_{i,p}$ are analogues of the $q$-derivations $e_i'$ and $\es_i$  defined  on the negative half $\calU_q^-(\g)$ of quantum group $\calU_q(\g)$ (Corollary \ref{cor: prime operation}).
We expect that $f_{i,p}$ and $\Ep_{i,p}$ may give ring-theoretic \emph{extended crystal operators} of the categorical crystal introduced in \cite{KP22}.  
Since $\qbA{} (\ttb)$ can be viewed as a generalization of the quantum unipotent coordinate ring $A_q(\n(w))$, it is natural and interesting to ask whether $\qbA{} (\ttb)$ possesses a quantum cluster algebra structure associated with $\ttb$. We conjecture that $\qbA{} (\ttb)$ has a quantum cluster algebra structure whose compatible pairs appear in \cite[\S 2]{FHOO2} even if $\ii$ is not locally reduced.  We also conjecture that, in the case of $\qbA{ } \simeq  \calK_t(\catC_{\hg}^0)$, there exists a subcategory $\catC_\ttb$ of the category $\catC_{\hg}^0$ which provides a \emph{monoidal categorification} of $\qbA{} (\ttb)$. Note that, when $\ii$ is a locally reduced sequence, this follows from the monoidal categorification in \cite{KKOP2}.  
It is also interesting to develop the \emph{global basis} (or \emph{canonical basis}) theory for $\hcalA(\ttb)$.

This paper is organized as follows. 
In Section \ref{Sec:Preliminaries}, we give necessary backgrounds on quantum groups, braid groups, PBW bases and etc.  
In Section \ref{Sec: Bosonic ext}, we define the bosonic extension $\qbA{\g}$ of a quantum group $\calU_q(\g)$ and investigate its properties. 
In Section \ref{Sec: bf and bg}, we introduce and study the bilinear form $\pair{ \ , \ }$ on $\qbA{\g}$, and investigate the braid group actions $\bT_i$.
In Section \ref{Sec: PBW locally reduced}, we develop the PBW theory associated with locally reduced sequences. 
In Section \ref{Sec: adjoint operators}, we introduce and study the adjoint operators $\Ep_{i,p}$ and $\Es_{i,p}$.
In Section \ref{Section: general PBW}, we develop the PBW theory associated with arbitrary sequences.

\medskip
\noindent
{\bf Acknowledgments}
We express our gratitude to Masaki Kashiwara, Myungho Kim, Yoshiyuki Kimura, Ryo Fujita and Hironori Oya for fruitful discussion and helpful comments.

\vskip 1em 

\subsection*{Convention}
Throughout this paper, we use the following convention.
\ben
\item For a statement $\ttP$, we set $\delta(\ttP)$ to be $1$ or $0$ depending on whether $\ttP$ is true or not. In particular, we set $\delta_{i,j}=\delta(i =j)$. 
\item For $a\in \Z \cup \{ -\infty \} $ and $b\in \Z \cup \{ \infty \} $ with $a\le b$, we set 
\begin{align*}
& [a,b] =\{  k \in \Z \ | \ a \le k \le b\}, &&  [a,b) =\{  k \in \Z \ | \ a \le k < b\}, \allowdisplaybreaks\\
& (a,b] =\{  k \in \Z \ | \ a < k \le b\}, &&  (a,b) =\{  k \in \Z \ | \ a < k < b\},
\end{align*}
and call them \emph{intervals}. 
When $a> b$, we understand them as empty sets. For simplicity, when $a=b$, we write $[a]$ for $[a,b]$. For an interval $[a,b]$, we set
$A^{[a,b]}$ to be the product of copies of a set $A$ indexed by $[a,b]$, and 
$$
\Z_{\ge0}^{\oplus [a,b]} \seteq \{ (c_a,\ldots,c_b) \ | \ c_k \in \Z_{\ge 0} \text{ and  $c_k=0$ except for finitely many $k$'s} \}. 
$$
We define $A^{[a,b)}$,  $\Z_{\ge 0}^{\oplus [a,b)}$, ..., etc in a similar way. 
\item For a totally ordered set $J=\{\cdots < j_{-1} < j_{0} < j_{1} < j_{2} < \cdots \}$, write
$$
\Opd_{j \in J}^{\longrightarrow} A_{j} \seteq \cdots A_{j_{2}}A_{j_1}A_{j_{0}}A_{j_{-1}} \cdots, \qquad 
\Opd_{j \in J}^{\longleftarrow} A_{j} \seteq \cdots A_{j_{-1}}A_{j_0}A_{j_{1}}A_{j_{2}} \cdots.
$$
\ee

\vskip 2em

\section{Preliminaries} \label{Sec:Preliminaries}

We briefly recall some backgrounds on quantum groups $\calU_q(\g)$ and $q$-boson algebras $\calB_q(\g)$ (see \cite{K91, Kimura12, LusztigBook} and references therein for details).  

Let $\sfC=(c_{i,j})_{i,j \in I}$ be the \emph{finite type Cartan matrix} and let $\g$ be the finite dimensional simple Lie algebra associated with $\cmC$. We denote
by $\Pi=\{\al_i \}_{i \in I}$ the set of simple roots, by $\Pi^\vee=\{ h_i \}_{i \in I}$ the set of simple coroots. 
We also denote by $\wl$ the weight lattice of $\g$, 
and by $\Phi$ the set of roots of $\g$. We write $\Phi_\pm$ for the set of positive (resp. negative) roots.  We call the quintuple $(\sfC,\wl,\Pi,\wl^\vee,\Pi^\vee)$ the Cartan datum of $\g$.

The free abelian group $\rl = \bigoplus_{i \in I} \Z \al_i$ is called the \emph{root lattice} and we set $\rl_{\pm} \seteq \pm \sum_{i \in I}\Z_{\ge0}\al_i$. For an element $\be = \sum_{i \in I} a_i\al_i \in \rl_\pm$, we define 
\begin{align} \label{eq: height and obe}
\het(\be) = \sum_{i \in I}|a_i|  \quad \text{ and } \quad  
\abs{\be} = \sum_{i \in I} |a_i|\al_i \in \rl_+.    
\end{align}
For $\al,\be \in \Phi$, we set 
\begin{align} \label{eq: p al,be}
p_{\be,\al} \seteq \max\{ p \in \Z \ | \ \be - p\al \in \Phi \}.     
\end{align}
There exists a diagonal matrix $\sfD=\diag(\sfd_i \in \Z_{\ge 1})_{i\in I}$ such that
$\sfD\sfC$ is symmetric. In this paper, we take such $\sfD$ satisfying $\min(\sfd_i)_{i \in I}=1$. There exists a symmetric bilinear form $( \cdot, \cdot)$ on $\wl$
such that 
\begin{align*}
    (\al_i,\al_j) =\sfd_i c_{i,j} \qtq \lan h_i,\al_j \ran = c_{i,j} = \dfrac{2(\al_i,\al_j)}{(\al_i,\al_i)} \qt{for any $i,j\in I$}. 
\end{align*}
Note that $\sfd_i = (\al_i,\al_i)/2 \in \Z_{\ge1}$ for $i\in I$.

\subsection{Quantum groups}

Let $q$ be an indeterminate and $q^{1/2}$ be the formal square root of $q$. We set 
$\bfk \seteq \Q(q^{1/2})$, 
$$q_i\seteq q^{(\al_i,\al_i)/2}, \quad  [n]_{q_i}= \dfrac{q_i^n-q_i^{-n}}{q_i-q_i^{-1}}, \quad [n]_{q_i}! = \prod_{k=1}^n [k]_{q_i} \qtq
\begin{bmatrix}
n \\ m
\end{bmatrix}_{q_i} = \dfrac{[n]_{q_i}!}{[m]_{q_i}![n-m]_{q_i}!},
$$
where $i\in I$ and $n \ge m \in \Z_{\ge0}$. We write $[n]_i$ (resp.\ $[n]$) instead of $[n]_{q_i}$ (resp.\ $[n]_q$) for simplicity. 

\begin{definition}
The \emph{quantum group} $\calU_q(\g)$ associated with the Cartan datum $(\sfC,\wl,\Pi,\wl^\vee,\Pi^\vee)$ is the $\bfk$-algebra
generated by $e_i,f_i$ $(i \in I)$ and $q^h$ $(h \in \wl^\vee)$ subject to the following relations:
\begin{align*}
 q^0=1, \quad q^{h}q^{h'}=q^{h+h'}, \quad & q^he_iq^{-h} = q^{\lan h,\al_i\ran}e_i, \quad q^hf_iq^{-h} = q^{-\lan h,\al_i\ran}f_i, \allowdisplaybreaks\\
 e_if_j - f_je_i &= \delta_{i,j} (t_i-t_i^{-1})(q_i-q_i^{-1})^{-1}, \allowdisplaybreaks\\
 \sum_{k=0}^{1-c_{i,j}} (-1)^k \begin{bmatrix}  1-c_{i,j} \\ k \end{bmatrix}_i e_i^{1-c_{i,j}-k}e_je_i^k 
&= \sum_{k=0}^{1-c_{i,j}} (-1)^k \begin{bmatrix}  1-c_{i,j} \\ k \end{bmatrix}_i f_i^{1-c_{i,j}-k}f_jf_i^k=0 \quad \text{for $i\ne j$}, 
\end{align*}
where $t_i \seteq q^{\sfd_ih_i}$. 
\end{definition}
Let $\calU_q^{\pm}(\g)$ be the subalgebra of $\calU_q(\g)$ generated by $e_i$'s (resp. $f_i$'s). 
We set $f_i^{(k)}\seteq f_i / [k]_i!$ and 
$\wt(x):=\be$ for any $x \in \calU_q^-(\g)_\be$, 
where $\calU_q^-(\g)_\be$ is the weight space of $\calU_q^-(\g)$ with weight $\beta \in \rl_-$.
We have the following automorphisms:
\bna
\item The $\bfk$-algebra anti-involution $*:\calU_q(\g) \to \calU_q(\g)$ given by
$$
e_i^* = e_i, \quad f_i^*=f_i \qtq (q^{h_i})^* = q^{-h_i}.
$$
\item For an automorphism $\sigma$ of the Dynkin diagram of $\g$, the $\bfk$-algebra automorphism
$\sigma: \calU_q(\g) \to \calU_q(\g)$ given by
$$
\sigma(e_i) = e_{\sigma(i)}, \quad \sigma(f_i) = f_{\sigma(i)} \qtq 
\sigma(q^{h_i}) = q^{h_{\sigma(i)}}. 
$$
\item The $\Q$-algebra anti-automorphism $-:\calU_q(\g)\to \calU_q(\g)$ given by
$$
\overline{e_i}=e_i, \quad \overline{f_i}=f_i, \quad \overline{q^{h_i}}=q^{h_i} \qtq \overline{q^{1/2}}=q^{-1/2}.
$$
\ee
 
\subsection{Braid groups and Weyl groups}

We denote by $\ttB_\sfC$  the \emph{$($generalized$)$ braid group} or \emph{Artin-Tits group} associated with the Cartan matrix $\cmC= (c_{i,j})_{i,j\in I}$. The braid group $\ttB_\sfC $ is given by generators $r_i^{\pm1}$ $(i \in I)$ and the defining relations:
\begin{align} \label{Eq: braid gp}
r_i r_i^{-1} = r_i^{-1}r_i = {\rm id} 
\qtq
\underbrace{r^{\pm1}_ir^{\pm1}_jr^{\pm1}_ir^{\pm1}_j \cdots}_{m(i,j) \text{ factors}} = \underbrace{r^{\pm1}_jr^{\pm1}_ir^{\pm1}_jr^{\pm1}_i \cdots}_{m(i,j) \text{ factors}} \quad \text{ for $i \ne j \in I$},
\end{align}
where 
\begin{align} \label{Eq: m(i,j)}
m(i,j) \seteq
\begin{cases}
c_{i,j} c_{j, i}+2 & \text{ if } c_{i,j} c_{j, i} \le 2,\\
6 & \text{ if } c_{i,j} c_{j, i} = 3.
\end{cases}
\end{align}
We call the second relation in~\eqref{Eq: braid gp} the \emph{braid relation} or \emph{$m(i,j)$-braid move}. We denote by $\ttB_\sfC^\pm$ the submonoid of $\ttB_\sfC$ generated by $r_i$ (resp. $r_i^{-1}$) for all $i\in I$. 

\smallskip 

Let
$\weyl_\sfC$ be the Weyl group associated with the Cartan matrix $\sfC$ generated by the simple reflections $s_i$ $(i \in I)$, where
$$ s_i\la = \la -\lan h_i,\la \ran\al_i \quad \text{for $\la \in \wl$}.  $$
The generators $s_i$ satisfy the following relations:
\begin{align*} 
  \ s_i^2 = {\rm id} \qtq  \ \underbrace{s_is_js_is_j \cdots}_{m(i,j) \text{ factors}} = \underbrace{s_js_is_js_i \cdots}_{m(i,j) \text{ factors}} \quad \text{ for $i\ne j \in I$}. 
\end{align*}
We then have the natural surjection   
\begin{align} \label{Eq: b->W}
\pi^{\pm} : \bg_\sfC^{\pm} \twoheadrightarrow \weyl_\sfC
\end{align}
sending $ r_{i}^{\pm}$ to $s_i$ for $i\in I$, respectively.

For an element $\ttb \in \bg_\cmC$ (resp.\ $w \in \weyl_\sfC$), we write $\ell(\ttb)$ (resp.\ $\ell(w)$) for its length. 
\begin{definition} \label{def: reduced} Let $\ttb\in \bg_\cmC^+$ and $w \in \weyl_\cmC$. Let $l\seteq \ell(\ttb)$ and $l'\seteq \ell(w) $.
\bnum
\item A sequence $\ii=(i_1,i_2,\ldots,i_l) \in I^{[1,l]}$ is said to be a \emph{expression} (or \emph{sequence}) of $\ttb$ if $\ttb = r_{i_1}r_{i_2}\cdots r_{i_l}$. We denote by $\Seq(\ttb)$ the set of all expression of $\ttb$.
\item A sequence $\ii=(i_1,\ldots,i_{l'}) \in I^{[1,l']}$ is said to be a \emph{reduced expression} (or \emph{reduced sequence}) of $w$
if $w = s_{i_1}\cdots s_{i_{l'}}$ and $ \ell(w) = l' $. We denote by $R(w)$ the set of all reduced expression of $w$.     

\ee
\end{definition}

Let $w_\circ$ be the longest element in $\weyl_\sfC$ and set $\ell \seteq \ell(w_\circ)$.
We define $\gar^\pm$ to be the element in $\ttB_\sfC^\pm$ such that $\ell(\gar^\pm)=\ell(w_\circ)$ and $\pi^\pm(\gar^\pm)=w_\circ$. 
For any sequence $\ii=( i_1,\ldots,i_l) \in I^l$, we set 
$$
\ttb_\ii\seteq r_{i_1}\cdots r_{i_l} \in \bg_\sfC^+ \quad \text{and}\quad 
w_\ii\seteq s_{i_1}\cdots s_{i_l} \in \weyl_\sfC.
$$
Note that $\Seq(\ttb) = \{ \ii \in I^{\ell(\ttb)} \ | \ \ttb_\ii =\ttb \}.$
We call a sequence $\ii \in I^N$ \emph{locally reduced} if the following condition is satisfied: For any $1 \le k \le l \le N$ such that $l-k+1 \le \ell(w_\circ)$, there exists $w \in \weyl$ such that
\begin{align} \label{eq: locally reduced}
(i_k,\ldots,i_{l}) \in R(w).       
\end{align}
For convenience to distinguish, we use $\bfi$ for locally reduced sequences and $\ii$ for general sequences.

\subsection{PBW bases for $\calU_q^-(\g)$}
Let $\bfi_\circ=(i_1,\ldots,i_\ell) \in R(w_\circ)$.
By setting $\rvt{\bfi_\circ}{k} \seteq s_{i_1} \ldots s_{i_{k-1}}(\al_{i_k})$ for each $1 \le k \le \ell$, we
have $\Phi_+ = \{ \rvt{\bfi_\circ}{k}  \ | \ 1 \le k \le \ell\}$. 
We define 
$$
\rvt{\bfi_\circ}{k} \le_{\bfi_\circ} \rvt{\bfi_\circ}{l} \quad \text{ for any } k \le l,
$$ 
which gives an 
 total order $\le_{\bfi_\circ}$ on $\Phi_+$. Note that the total order $\le_{\bfi_\circ}$ is \emph{convex} in the sense that,
for $\al,\be \in \Phi_+$ with $\al+\be \in \Phi_+$, either $\al \le_{\bfi_\circ} \al+\be \le_{\bfi_\circ} \be$ or $\be \le_{\bfi_\circ} \al+\be \le_{\bfi_\circ} \al$.
Let $(\al,\be)$ be a pair of positive roots such that $ \al <_{\bfi_\circ} \be $ and $\eta \seteq \al+\be \in \Phi_+$.
We call the pair $(\al,\be)$ an \emph{$\bfi_\circ$-minimal pair of $\eta$} if 
there exists no pair $(\al',\be') $ of positive roots such that $\al'+\be'=\eta$ and 
$\al <_{\bfi_\circ} \al' <_{\bfi_\circ} \eta <_{\bfi_\circ} \be' <_{\bfi_\circ} \be$.

\smallskip

Recall the braid symmetries and PBW bases for $\calU_q^-(\g)$ by following \cite{LusztigBook}. For $i\in I$, we set $\bS_i\seteq T'_{i,-1} $ and $\bS_i^{-1}\seteq T''_{i,1}$, where $T'_{i,-1}$ and $T''_{i,1} $ are Lusztig's braid symmetries defined in \cite[Chapter 37]{LusztigBook}.
Note that 
\begin{equation} \label{eq: Si}
\begin{aligned}
& \bS_i(f_i)= -e_it_i, \quad \quad \   \bS_i(f_j) = \sum_{r+s = -c_{i,j}}  (-1)^{ r} q_i^{r} f_{i}^{(s)} f_{j} f_{i}^{(r)} \text{ for $i \ne j$}, \\ 
&\bS_i^{-1}(f_i)= -t_i^{-1}e_i, \quad  \bS_i^{-1}(f_j) = \hs{-1ex} \sum_{r+s = -c_{i,j}} \hs{-1ex}  (-1)^{ r} q_i^{r} f_{i}^{(r)} f_{j} f_{i}^{(s)} \text{ for $i \ne j$}.
\end{aligned}
\end{equation}
Then the automorphisms $\{ \bS_i^{\pm1} \}$ satisfy the braid relations and hence $\bg_\sfC$ acts on $\calU_q(\g)$ via $\bS_i^\pm$. 
\smallskip

Let us fix a reduced sequence $\bfi_\circ=(i_1,\ldots,i_\ell)$ of $w_\circ$. For $1\le k \le \ell$, define 
\begin{align} \label{eq: pbw vector U}
\calF^{\bfic}_k \seteq \bS_{i_1} \bS_{i_2} \cdots \bS_{i_{k-1}} (f_{i_k}).     
\end{align}
Note that $\calF^{\bfic}_k \in \calU_\bbA^-(\g)$ and $\wt( \calF^{\bfic}_k ) =-\rvt{\bfi_\circ}{k}$.
We drop ${\bfic}$ from the notation if there is no danger of confusion. 
We call $\calF^{\bfic}_k$ the \emph{$k$-th PBW root vector} corresponding to $\bfic$. For each $\bsc = (c_1,\ldots,c_\ell) \in \Z_{\ge0}^\ell$, we define 
$$
\calF_{\bfic}(\bsc) \seteq \calF^{(c_\ell)}_\ell   \calF^{(c_{\ell-1})}_{\ell-1}   \cdots 
\calF^{(c_1)}_1,
$$
where $\calF_k^{(a)}  \seteq  \calF^{a}_k  /[a]_{i_k}!$ for $a \in \Z_{\ge0}$ and $k \in [1,\ell]$. 
It was shown that the set
\begin{align} \label{eq: PBW basis}
\calP_{\bfic} \seteq \{ \calF_{\bfic}(\bsc) \ | \ \bsc \in \Z_{\ge 0}^\ell \}    
\end{align}
forms a basis of $\calU_q^-(\g)$ (\cite{L90, LusztigBook}). We call $\calF_{\bfic}(\bsc)$ a \emph{PBW monomial} and $\calP_{\bfi_\circ}$ the \emph{PBW basis} of $\calU_q^-(\g)$ corresponding to $\bfic$.

\subsection{Bilinear forms and $q$-boson algebras} \label{Subsection: bf and qba}
By \cite[\S 1.2.5]{LusztigBook}, there exists a unique non-degenerate symmetric bilinear form $( \cdot , \cdot )_L$ on $\calU_q^-(\g)$ satisfying
$$
(1,1)_L = 1, \qquad (f_i,f_j)_L = \dfrac{\delta_{i,j}}{1-q_i^2}
$$
(see \cite{LusztigBook} for details and see also \cite[\S 2.1]{Kimura12}).
For $i \in I$, there exists a unique $\bfk$-linear endomorphism $e'_i$ and $\es_i$ of $\calU^-_q(\g)$ such that
\begin{equation}\label{Eq: ei ei*}
\begin{aligned}
& e'_i(f_j) = \delta_{i,j}, &&  e'_i(xy) = e'_i(x)y + q_i^{\lan h_i,\wt(x) \ran}x e'_i(y), \\   
& \es_i(f_j) = \delta_{i,j}, &&  \es_i(xy) = x  \es_i  (y) + q_i^{\lan h_i,\wt(y) \ran}  \es_i  (x) y,
\end{aligned}
\end{equation}
where $x,y \in \calU^-_q(\g)$ are homogeneous elements (see \cite[\S 3]{LusztigBook}, \cite{K91} and see also \cite[Lemma 8.2.1]{KKKO18}). Note that 
\begin{align} \label{eq: bracket}
[e_i,x] = \dfrac{\es_i(x)t_i - t_i^{-1}e'_i(x)}{q_i-q_i^{-1}}    
\end{align}
for $x \in \Um$ and $i \in I$.
\begin{lemma} For elements $y,z \in \Um$ with $\es_i(y)=0$ and $e'_i(z)=0$, we have
\begin{align} \label{Eq: es e'i for f_i^m}
e'_i(  f_{i}^{m} z ) =   q_i^{-m+1}[m]_i f_{i}^{m-1} z, \qquad \es_i( y f_{i}^{m} ) =   q_i^{-m+1}[m]_i  y f_{i}^{m-1} 
\end{align}
for any $m \in \Z_{\ge 0}$.
\end{lemma}
\begin{proof}
The assertion follows from~\eqref{Eq: ei ei*}.    
\end{proof}

The braid symmetry $\bS_i$ induces a $\bfk$-algebra isomorphism $\bS_i : {_i}\calU \isoto \calU_i$, where
\begin{align} \label{eq: Ui}
_i\calU \seteq \{ x \in \Um \ | \ e_i'(x)=0 \} \quad \text{ and } \quad 
\calU_i \seteq \{ x \in \Um \ | \ \es_i(x)=0 \}.
\end{align}
It is known that 
$$ 
( x, y)_L = ( \bS_i(x), \bS_i(y))_L 
$$ 
for any homogeneous $x,y \in {_i\calU}$ (see \cite[Proposition 38.2.1]{LusztigBook} and see also \cite[\S 4.5]{Kimura12}).

In \cite[\S 3.4]{K91}, Kashiwara proved that there exists a unique non-degenerate symmetric pairing $( \cdot , \cdot)_K$ on $\Um$ satisfying
\begin{align*} 
(1,1)_K = 1, \qquad (f_ix,y)_K = (x,e'_i(y))_K.    
\end{align*}
It is proved in \cite[\S 2.2]{L04} that
\begin{align} \label{eq: Kform Lform}
(x,y)_K = \prod_{i \in I} (1-q_i^2)^{n_i}(x,y)_L    
\end{align}
for homogeneous elements $x,y \in \Um_{\be}$ with $\be = -\sum n_i\al_i \in \rl_-$.

Let $\calB_q(\g)$ be the \emph{$q$-boson algebra} (also known as the \emph{Kashiwara algebra}) of $\g$, which is isomorphic to the $\bfk$-subalgebra of $\End(\Um)$ generated by $\{ \sfe_i',\sff_{j} \}_{i,j\in I}$, where  $\sff_{j}$ is the left multiplication of $f_j$ and $\sfe'_i \seteq e'_i$ in $\End(\Um)$.
Note that the generators $\{ \sfe_i',\sff_{j} \}_{i,j\in I}$ satisfy the defining relations
\begin{align*}
 \sfe_i' \sff_{j}  &= q_i^{-\lan h_i,\al_j \ran} \sff_j\sfe'_i +\delta_{i,j}, \text{ for all $i,j \in I$, and } \\
 \sum_{k=0}^{1-c_{i,j}} (-1)^k \begin{bmatrix}  1-c_{i,j} \\ k \end{bmatrix}_i & {\sfe'_i}^{1-c_{i,j}-k}\sfe'_j{\sfe'_i}^k 
= \sum_{k=0}^{1-c_{i,j}} (-1)^k \begin{bmatrix}  1-c_{i,j} \\ k \end{bmatrix}_i \sff_i^{1-c_{i,j}-k}\sff_j\sff_i^k=0 \text{ for $i\ne j$}
\end{align*}
(see \cite[\S 3.4]{K91}).

\begin{lemma} \label{Lem: Py}
Let $y \seteq \bS_i(f_i)= -e_i t_i$ and $P \in \calU_q^-(\g)$ with $\es_i(P)=0$. Then
$$
Py = q_i^{ - \langle  h_i, \wt(P) \rangle} \left(  y P + \frac{1}{  q_i^2 ( q_{i}^{-1} - q_i ) }     e'_i(P)  \right).
$$
\end{lemma}	
\begin{proof}
It follows from \eqref{eq: bracket} that $ e_i P - P e_i = \frac{1}{  q_{i}^{-1} - q_i } t_i^{-1} e'_i(P) $, i.e., 
$$
P e_i = e_i P -  \frac{1}{  q_{i}^{-1} - q_i } t_i^{-1} e'_i(P).
$$
We thus have 
\begin{align*}
Py &= -P e_i t_i	= - \left(e_i P -  \frac{1}{  q_{i}^{-1} - q_i } t_i^{-1} e'_i(P) \right) t_i \allowdisplaybreaks\\ 
&= - q_i^{ - \langle  h_i, \wt(P) \rangle} e_i t_i P + \frac{1}{  q_{i}^{-1} - q_i } q_i^{ - \langle  h_i, \wt(P)+\al_i \rangle} e'_i(P) \allowdisplaybreaks\\
&= q_i^{ - \langle  h_i, \wt(P) \rangle} \left(  yP +  \frac{1}{ q_i^2  \left( q_{i}^{-1} - q_i \right) } e'_i(P) \right). \qedhere 
\end{align*}		
\end{proof}	

\vskip 2em 

\section{Bosonic extensions of quantum groups} \label{Sec: Bosonic ext}
In this section, we investigate several properties of the \emph{bosonic extensions of quantum groups} introduced in \cite{HL15} (see also \cite{FHOO}) for simply-laced finite type and in \cite{JLO2} for arbitrary finite type.

Let $\cmC = (c_{i,j})_{i,j\in I}$ be a finite Cartan matrix and let $\calU_q(\g)$ be the quantum group associated with $\cmC$. 
We set 
$$\hI \seteq I \times \Z.$$

\begin{definition} [{see \cite[Theorem 7.3]{HL15}, \cite[\S 7.1]{JLO2}}] \label{Def: extended qg}
The \emph{bosonic extension} $\hcalA_\g$ of the quantum group $\calU_q(\g)$ is the $\bR$-algebra generated by $\{ f_{i,p} \}_{ (i,p) \in \hI }$
satisfying the following defining relations: for any $i,j \in I$ and $m,p\in \Z$,
\bna
\item  $\displaystyle \sum^{1-c_{i,j}}_{k=0} (-1)^k \left[\begin{matrix}1-c_{i,j} \\ k\\ \end{matrix} \right]_i {f_{i,p}}^{1-c_{i,j}-k} f_{j,p} f_{i,p}^{k} = 0 \quad \text{ for }  i \ne j $,
\item \label{it: relation2} $f_{i,m} f_{j,p} = q_i^{(-1)^{p-m+1}c_{i,j}} f_{j,p} f_{i,m}$  for $p> m+1$,
\item \label{it: relation3} $f_{i,p} f_{j,p+1} = q_i^{ c_{i,j} } f_{j,p+1} f_{i,p} + \delta_{i,j} (1-q_i^2)$.
\ee
\end{definition}
If there is no danger of confusion, we drop $\g$ from the notation $\hcalA_\g$ for simplicity. 

\begin{remark} \label{Rem: q and t} \ 
The $t$-deformed Grothendieck ring of the Hernandez-Leclerc category of a quantum affine algebra $\calU_q'(\hg)$ 
is isomorphic to the bosonic extension $\hcalA_\gf$ of the quantum group $\calU_q(\gf)$ (\cite[Theorem 7.3]{HL15} and  \cite[Theorem 10.4]{FHOO}). Here $\gf$ is the simple Lie algebra of finite simply-laced type given in \cite[Theorem 4.6]{KKOP22} corresponding to the quantum affine algebra   $\calU_q'(\hg)$. 
Note that $q^{1/2}$ in our paper corresponds to $t^{-1/2}$ in \cite[Theorem 7.3]{HL15}.
In Corollary~\ref{cor: q to qinverse} below, we will prove that
the presentation  in~\cite{HL15,JLO2} can be understood as relations among the adjoint operators of $\{ f_{i,p} \}_{(i,p)\in \hI}$.  
\end{remark}

For an interval $\ttJ$ of $\Z$, 
we define $\qbA{\ttJ}$ to be the subalgebra of $\qbA{}$ generated by $ f_{i,p} $ for $i\in I$ and $p\in \ttJ$. 
In particular, we set
$$
\qbA{\; \ge a} \seteq \qbA{[a,\infty]},  \quad  \qbA{\; > a} \seteq \qbA{(a,\infty]}, \quad
\qbA{\le b} \seteq \qbA{[-\infty,b]}, \quad \qbA{< b} \seteq \qbA{[-\infty,b)}.
$$
Defining a weight function by 
\begin{align} \label{Eq: weight}
\wt(f_{i,p}) \seteq (-1)^{p+1} \al_i \qquad  \text{ for any } (i,p) \in \hI,
\end{align}
the defining relations in Definition \ref{Def: extended qg} are homogeneous with respect to weights.
In particular, the second and third relations can be written in terms of weights as follows:
\begin{equation} \label{Eq: defining rel}
\begin{aligned} 
f_{i,m} f_{j,p} &= q^{-( \wt(f_{i,m}),\wt(f_{j,p}))} f_{j,p} f_{i,m}   \qquad  \qquad    \text{for $p> m+1$},   \\
f_{i,p} f_{j,p+1} &= q^{ -(\wt(f_{i,p}), \wt(f_{j, p+1})) } f_{j,p+1} f_{i,p} + \delta_{i,j} (1-q_i^2).
\end{aligned}
\end{equation}

The following lemmas say that $\qbA{}$ is viewed as an infinite tensor product of
$\qbA{[k]} \simeq \Um$'s.

\begin{lemma} \label{Lem: qA decomp} 
Let $ a,b \in \Z $ with $a \le b$.
\bnum
\item \label{it: Decom1} As a $\bR$-vector space, we have 
$$
\qbA{[a,b]} \simeq \qbA{[b]} \otimes_{\bR} \qbA{[b-1]} \otimes_{\bR} \cdots \otimes_{\bR}  \qbA{[a]}.  
$$
\item \label{it: Decom2} Let $B_k$ be a $\bR$-linear basis of $\qbA{[k]}$ for each $k\in [a,b]$. Then the set $ B_b \times B_{b-1} \times \cdots \times B_{a}$ is a $\bR$-linear basis of $\qbA{[a,b]}$.
\ee
\end{lemma}
\begin{proof}
Since~\eqref{it: Decom2} is a direct consequence of~\eqref{it: Decom1}, it suffices to prove~\eqref{it: Decom1}.
 Define a $\bR$-linear map 
$$ 
\phi: \qbA{[b]} \otimes_{\bR} \qbA{[b-1]} \otimes_{\bR} \cdots \otimes_{\bR}  \qbA{[a]} \longrightarrow \qbA{[a,b]}  
$$
by sending $ x_b \tens x_{b-1} \tens \cdots  \tens x_a  \mapsto x_{b}x_{b-1} \cdots x_{a}$ for $x_k \in \qbA{[k]}$ for $k\in [a,b]$. 
By the defining relations~\eqref{it: relation2} and~\eqref{it: relation3} of Definition \ref{Def: extended qg}, it is easy to see that $\phi$ is surjective. 

On the other hand, we fix a Dynkin quiver $Q$ of the Cartan matrix $\cmC$ and consider the isomorphism $\Theta_Q$  given in \cite[Definition 2.6, Theorem 7.2]{JLO2} (see \cite[Theorem 7.3]{HL15} for symmetric cases): 
$$
\Theta_Q:  \qbA{} \buildrel \sim \over \longrightarrow  \Q(q^{1/2}) \tens_{\Z[q^{\pm 1/2}]} \mathfrak{K}_q (\g),  
$$ 
where $\mathfrak{K}_q (\g)$ is the quantum virtual Grothendieck ring (see \cite[Definition 2.6]{JLO2}). We then have the commutative diagram
\begin{equation} \label{Eq: diagram for Kq(Cg)}
\begin{aligned} 
\xymatrix{
\qbA{[b]} \otimes_{\bR} \cdots \otimes_{\bR}  \qbA{[a]} \ar@{->>}[r]^{ \quad\qquad \phi} \ar[rd]_{ \Theta_Q \circ \phi }  & \qbA{[a,b]} \ar@{>->}[d]^{\Theta_Q|_{ \qbA{[a,b]}  }} \\
& \Q(q^{1/2}) \tens_{\Z[q^{\pm 1/2}]} \mathfrak{K}_q (\g)
}.
\end{aligned}
\end{equation}
It follows from the standard basis $\mathsf{E_q}$ of $\mathfrak{K}_q (\g)$  (see \cite[proof of Theorem 7.3]{HL15} for symmetric cases and \cite[Section 2.5]{JLO2} for general cases) that the composition $\Theta_Q \circ \phi$ is injective. This implies that $\phi$ is an isomorphism. 
\end{proof}

In the proof of Lemma \ref{Lem: qA decomp}, when $a=b$, the composition $\Theta_Q \circ \phi$ in the diagram \eqref{Eq: diagram for Kq(Cg)} gives an isomorphism between $\qbA{[a]} $ and the subalgebra $ \Q(q^{1/2}) \tens_{\Z[q^{\pm 1/2}]}\mathfrak{K}_{q, Q}  $ of $ \Q(q^{1/2}) \tens_{\Z[q^{\pm 1/2}]} \mathfrak{K}_q (\g)$.
Thus the following lemma follows from \cite[Theorem 6.1]{HL15} and \cite[Corollary 5.3]{JLO2}.

\begin{lemma} \label{Lem: id A and U}
For any $k\in \Z$, there is  a $\bfk$-algebra isomorphism 
$$
(\qbA{\g})_{[k]} \simeq \calU_q^-(\g)
$$ sending $ f_{i,k} $ to $f_i$ for any $i\in I$.  
\end{lemma}

\begin{definition} For an interval $[a,b]$,
we call $x \in \qbA{[a,b]} $ \emph{homogeneous} if $x$ can be written as $x=x_{b} x_{b-1} \cdots x_{a}$ such that each $x_k $ is contained in a weight space $(\qbA{[k]})_{\beta}$ of $\qbA{[k]}$ for some $\beta \in \rlQ^{\pm}$. 
\end{definition}
By Lemma \ref{Lem: qA decomp}, any element $x$ of $ \qbA{[a,b]}$ can be written as a sum of homogeneous elements, i.e., 
$$ 
x = \sum_{t}  x_{b, t} x_{b-1, t} \cdots x_{a, t}, 
$$ 
where $x_{k, t}$ is an element in a weight space of $\qbA{[k]}$ and $t$ runs over a finite index set.

\begin{remark} \label{Rmk: Aq[k] and Aq[k+1]} \
\bnum
\item Thanks to \eqref{Eq: weight}, one can consider weight spaces $((\qbA{\g})_{[k]})_\beta$ of $(\qbA{\g})_{[k]}$ of weight $\beta \in \rl_\pm$. Via the isomorphism $(\qbA{\g})_{[k]} \simeq \calU_q^-(\g)$ given in Lemma \ref{Lem: id A and U}, the weight space $((\qbA{\g})_{[k]})_\beta$ corresponds to $\calU_q^-(\g)_{(-1)^k\beta}$.
\item The operators $e_i'$ and $\es_i$ in \eqref{Eq: ei ei*} act on $(\qbA{\g})_{[k]}$ for any $k\in \Z$ via the isomorphism $(\qbA{\g})_{[k]} \simeq \calU_q^-(\g)$ given in Lemma \ref{Lem: id A and U}.
This will be used several times throughout the paper.
\item 
The negative half $\calU_q^-(\g)$ is a module over the $q$-boson algebra $\calB_q(\g)$ (see Section \ref{Subsection: bf and qba}). 
It is easy to show that, for any $k\in \Z$, there is a $\bfk$-algebra isomorphism  
\begin{align} \label{Eq: psi_p}
\psi_k : \calB_q(\g) \buildrel \sim \over \longrightarrow  \qbA{[k,k+1]}  
\end{align}
sending $ \bbf_i  \mapsto f_{i,k}  $ and $ \bbe'_i  \mapsto  \frac{q_i^2}{ q_i^{2}-1 } f_{i,k+1}$. 
\ee
\end{remark}

\begin{lemma} \label{Lem: x f_i,k+1}
For any $i\in I$ and $k\in \Z$, we have 
$$
 x f_{i,k+1} =  q_i^{ - \langle h_i,  \wt(x) \rangle  } ( f_{i,k+1}  x  +  q_i^{-1} ( q_i^{-1} - q_i)  e'_i (x) ) \qquad \text{ for any $x \in \qbA{[k]}$.} 
$$
\end{lemma}
\begin{proof}
Let $U$ be the subalgebra of the $q$-boson algebra $\calB_q(\g)$ generated by $ \bbf_i$ ($i\in I$), and for the sake of simplicity, we identify $U$ with $\calU_q^-(\g)$ via the map defined by $\bbf_i \mapsto f_i$ ($i\in I$). 
By \eqref{Eq: ei ei*}, we have 
$$
 \bbe_i' x = q_i^{ \langle h_i,  \wt(x) \rangle  } x \bbe'_i + e'_i (x)  \qquad \text{ for any $x\in U$.}
$$ 
Thus the assertion follows by pushing the above identity via the isomorphism $\psi_k$ given in \eqref{Eq: psi_p}. 
\end{proof}

From Definition \ref{Def: extended qg}, 
we have the following automorphisms on $\qbA{}$:
\bna
\item the $\bfk$-algebra anti-automorphism $*:\qbA{} \buildrel \sim \over \longrightarrow  \qbA{}$ is defined by
$$
 f_{i,p}^* = f_{i, -p}  \quad \text{for any $(i,p) \in \hI$,}
$$
\item for a Dynkin diagram automorphism $\sigma$ of $\g$, the $\bfk$-algebra automorphism $\sigma:\qbA{} \buildrel \sim \over \longrightarrow  \qbA{}$ is defined by
$$
\sigma(f_{i,p}) = f_{ \sigma(i), p} \quad \text{for any $(i,p) \in \hI$,}
$$
\item the $\Q$-algebra anti-automorphism $-:\qbA{} \buildrel \sim \over \longrightarrow  \qbA{}$ is given by
$$
\overline{f_{i,p}}=f_{i,p} \text{ for any $(i,p) \in \hI$} \qtq \overline{q^{1/2}}=q^{-1/2},
$$
\item the $\bfk$-algebra  automorphism $\dD : \qbA{} \buildrel \sim \over \longrightarrow  \qbA{} $ is given by 
$$\dD(f_{i,p}) = f_{i, p+1} \quad \text{ for any $(i,p) \in \hI$.}$$
\ee
One can check that
\begin{align*}
&\dD \circ \sigma = \sigma \circ \dD, \quad \dD \circ  - = - \circ \dD, 
\quad * \circ \dD =  \dD^{-1} \circ *, \\ 
& * \circ \ \sigma =  \sigma \circ *, \quad \ \  * \circ - =  - \circ *, \ \quad 
 \sigma \circ - =  - \circ \sigma.
\end{align*}

\vskip 2em 

\section{Braid group actions and bilinear forms} \label{Sec: bf and bg}

In this section, we investigate the
braid group actions on $\qbA{}$ introduced in~\cite{KKOP21A} for simply-laced finite type and in~\cite{JLO2} for arbitrary finite type. We then introduce and study a non-degenerate symmetric bilinear form on $\qbA{}$ which is invariant under the braid group actions. 

As the previous sections, let $\cmC = (c_{i,j})_{i,j\in I}$ be a finite Cartan matrix and let $\qbA{}$ be the bosonic extension associated with $\cmC$. 

\subsection{Braid group actions}
For any $i\in I$, the $\bfk$-automorphism $\bT_i: \qbA{} \longrightarrow \qbA{}$ defined by 
\begin{align*}
\bT_i(f_{j,p}) \seteq
\begin{cases}
f_{j,p+\delta_{i,j}} & \text{ if } c_{i,j} \ge 0,\\
\displaystyle \frac{   \sum_{r+s = -c_{i,j}}  (-1)^{ r} q_i^{ c_{i,j}/2 + r} f_{i,p}^{(s)} f_{j,p} f_{i,p}^{(r)}   }{ (q_i^{-1} - q_i)^{-c_{i,j}}} & \text{ if } c_{i,j}< 0,
\end{cases}
\end{align*}
is introduced in \cite[Theorem 2.3]{KKOP21A} for simply-laced finite types and \cite[Section 8.1]{JLO2} for non simply-laced finite types. 
The inverse of $\bT_i$ is given as follows: 
\begin{align*}
\bT_i^{-1}(f_{j,p}) =
\begin{cases}
f_{j,p-\delta_{i,j}} & \text{ if } c_{i,j} \ge 0,\\
\displaystyle \frac{ \sum_{r+s = -c_{i,j}}  (-1)^{ r} q_i^{  c_{i,j}/2 + r} f_{i,p}^{(r)} f_{j,p} f_{i,p}^{(s)}   }{ (q_i^{-1} - q_i)^{-c_{i,j}}  } & \text{ if } c_{i,j} < 0.
\end{cases}
\end{align*}

\begin{lemma} \label{Lem: Ti and others}
For any $i\in I$, we have the followings$:$ 
\bnum
\item \label{it: d T} $ \dD \circ \bT_i = \bT_i \circ \dD $.
\item \label{it: star T} $\bT_i^{-1} = * \circ \bT_i \circ * $.
\item \label{it: sigma T} $\sigma \circ \bT_i =  \bT_{\sigma(i)} \circ \sigma $ for any Dynkin diagram automorphism $\sigma$ of $\g$.
\item \label{it: bar T} $\bT_i \circ - =   - \circ  \bT_i $.
\ee
\end{lemma}
\begin{proof}
Since~\eqref{it: d T},~\eqref{it: star T} and~\eqref{it: sigma T} follow from the definitions directly, we only deal with~\eqref{it: bar T}. 

As the case where $i=j$ is obvious, we assume that $i\ne j$. Then we have 
\begin{align*}
\overline{\bT_i(f_{j,p})} &=  
\frac{   \sum_{r+s = -c_{i,j}}  (-1)^{ r} q_i^{ -c_{i,j}/2 - r} f_{i,p}^{(r)} f_{j,p} f_{i,p}^{(s)}   }{ (q_i - q_i^{-1})^{-c_{i,j}}} \allowdisplaybreaks\\
& = 
\frac{   \sum_{r+s = -c_{i,j}}  (-1)^{ s} q_i^{ c_{i,j}/2 + s} f_{i,p}^{(r)} f_{j,p} f_{i,p}^{(s)}   }{ (q_i^{-1} - q_i)^{-c_{i,j}}} \allowdisplaybreaks\\
&= \bT_i(f_{j,p}),
\end{align*}
which implies~\eqref{it: bar T}.
\end{proof}

\begin{proposition} [{\cite[Theorem 2.3]{KKOP21A} and \cite[Thoerem 8.1]{JLO2}}] \label{Prop: Ti braid}
The automorphisms $\bT^{\pm1}_i$ $(i\in I)$ satisfy the braid group relations for $\ttB_\cmC$.
\end{proposition}
For any $\ii = (i_1, \ldots, i_t) \in I^t$, we set 
$$
\bT_\ii \seteq \bT_{i_1} \bT_{i_2} \cdots \bT_{i_t}.
$$
Thanks to Proposition \ref{Prop: Ti braid}, the followings are well-defined:
\begin{align*}
\bT_{\ttb^{\pm}} & \seteq \bT^{\pm1}_{i_1} \bT^{\pm1}_{i_2} \cdots \bT^{\pm1}_{i_n} && \text{ for any $(i_1,\ldots,i_n) \in \Seq(\ttb^{\pm})$  of $\ttb^\pm \in \ttB_\sfC^{\pm}$}, \\
\bT_{w} & \seteq \bT_{j_1} \bT_{j_2} \cdots \bT_{j_r} && \text{ for any $(j_1,\ldots,j_r) \in R(w)$ of $w\in \weyl$}.
\end{align*}

For any $i\in I$, we set 
\begin{align} \label{Eq: const}
\cC_i \seteq q_i^{1/2} (q_i^{-1} - q_i).
\end{align}
We simply write $\cC $ for $ \cC_i$ when $\cmC$ is of symmetric type.
If $\cmC$ is of non-symmetric type, we often write $\cC_s \seteq \cC_i$ (resp.\ $\cC_l \seteq \cC_i$) when $\al_i$ is short (resp.\ long).  
Note that, for $i,j\in I$ with $c_{i,j} < 0$,
\begin{equation} \label{Eq: Ti fjp}
\begin{aligned} 
\bT_i(f_{j,p}) =  \cC_i^{c_{i,j}} \sum_{r+s = -c_{i,j}}  (-1)^{ r} q_i^{r} f_{i,p}^{(s)} f_{j,p} f_{i,p}^{(r)}, \\
\bT_i^{-1}(f_{j,p}) =  \cC_i^{c_{i,j}} \sum_{r+s = -c_{i,j}}  (-1)^{ r} q_i^{r} f_{i,p}^{(r)} f_{j,p} f_{i,p}^{(s)}.
\end{aligned}
\end{equation}

\begin{proposition} \label{Prop: T and S} 
Let $k\in \Z$.
Under the identification $(\qbA{\g})_{[k]} \simeq \calU_q^-(\g)$ given in {\rm Lemma \ref{Lem: id A and U}}, we have 
$$
\bT_i(x) = \cC_i^{  \langle h_i,  \abs {\wt(x)} \rangle } \bS_i(x)
$$
for any homogeneous element $x \in (\qbA{\g})_{[k]}$ with $  e'_i(x)=0$. Here, $\abs{\beta}$ is defined in \eqref{eq: height and obe}.
\end{proposition}
\begin{proof}
Since $\bT_i$ commutes with $\dD$, we may assume that $k=0$. 
We denote by $\phi : \calU_q^-(\g) \buildrel \sim \over \longrightarrow \qbA{[0]} $ the isomorphism given in Lemma \ref{Lem: id A and U}, and set ${_i}\mathcal{A} \seteq \phi ({_i}\calU) $ and  $\mathcal{A}_i \seteq \phi (\calU_i) $, where ${_i}\calU$ and $\calU_i$ are given in \eqref{eq: Ui}.
If no confusion arises, we identify $\calU_q^-(\g)$ with $\qbA{[0]}$ via $\phi$.

Let $ \bfic = (i_1, i_2, \ldots, i_\ell) \in R(w_\circ) $ with $i_1 = i$, and set $\bficp \seteq (i_2, \ldots, i_\ell, i^*) $. Note that $\bficp \in R(w_\circ)$. For $t=1,2, \ldots, \ell$, define  
\begin{align*}
y_t \seteq \calF^{\bfic}_t, \quad \gamma_t \seteq - \wt(y_t)= \rvt{\bfic}{t} \quad \text{ and } \quad 
x_t \seteq  \calF^{\bficp}_t, \quad \beta_t \seteq -\wt(x_t)= \rvt{\bficp}{t}. 
\end{align*}
where $\calF^{\bfic}_t$ is defined in \eqref{eq: pbw vector U}. Note that $x_\ell = y_1 = f_{i}$ and $ y_{t+1} = \bS_i (x_t)$ for $t=1,2, \ldots, \ell-1$. 
It is known that 
\bna
\item ${_i}\mathcal{A}$ is generated by $x_1$, $x_2$, \ldots, $x_{\ell-1}$,  
\item $\mathcal{A}_i$ is generated by $y_2$, $y_3$, \ldots, $y_\ell$.   
\ee
Thus, it suffices to show that  
$$
\bT_i(x_t) = \cC_i^{ \langle h_i, \beta_t \rangle} y_{t+1} \quad \text{ for } t=1,2, \ldots, \ell-1.
$$
We shall use induction on $\Ht(\be_t)$. 

If $\Ht(\be_t)=1$, then it is obvious by \eqref{eq: Si} and \eqref{Eq: Ti fjp}. 

We assume that $\Ht(\be_t)>1$, and let $( \be_a, \be_b )$  be a minimal pair of $\be_t$. By \cite[Theorem 4.2 (4.4)]{BKM12}, we have 
\begin{align} \label{Eq: xaxb}
x_ax_b - q^{-(\beta_a, \beta_b )} x_b x_a =  [p_{\beta_b,\beta_a}+1] x_t,
\end{align}
where $p_{\beta_b,\beta_a}$ is the integer given in~\eqref{eq: p al,be}. 
Note that $\bT_i(x_a) = \cC_i^{  \langle h_i, \beta_a \rangle } y_{a+1} $ by the induction hypothesis.

If $ b \ne \ell$, i.e. $x_b \ne f_i$, then the induction hypothesis tells us that $\bT_i(x_b) = \cC_i^{  \langle h_i, \beta_b \rangle } \bS_i(x_b) $. Applying $\bT_i$ to the equality \eqref{Eq: xaxb}, we have 
\begin{align*}
[p_{\beta_b,\beta_a}+1]  \bT_i(x_t) &= \bT_i(x_a) \bT_i(x_b) - q^{-(\beta_a, \beta_b )} \bT_i(x_b) \bT_i(x_a) \allowdisplaybreaks\\
&=  \cC_i^{  \langle h_i, \beta_a + \beta_b \rangle } \left( \bS_i(x_a) \bS_i(x_b) - q^{-(\beta_a, \beta_b )} \bS_i(x_b) \bS_i(x_a) \right) \allowdisplaybreaks\\
&=  \cC_i^{  \langle h_i, \beta_t \rangle } \left(  y_{a+1} y_{b+1} - q^{-(\gamma_{a+1}, \gamma_{b+1} )} y_{b+1} y_{a+1} \right)\allowdisplaybreaks \\
&= \cC_i^{  \langle h_i, \beta_t \rangle }  [p_{\gamma_{b+1},\gamma_{a+1}}+1] y_{t+1} \\ &= \cC_i^{  \langle h_i, \beta_t \rangle }  [p_{\gamma_{b+1},\gamma_{a+1}}+1] \bS_i(x_{t}),
\end{align*}
where the third equality follows from the fact that $ (\beta_a, \beta_b) = (\gamma_{a+1}, \gamma_{b+1}) $, and the fourth equality follows from the fact that $( y_{a+1}, y_{b+1})$ is a minimal pair of $y_{t+1}$. Since $p_{\beta_b,\beta_a} = p_{\gamma_{b+1},\gamma_{a+1}} $, we have $\bT_i(x_t) = \cC_i^{  \langle h_i, \beta_t \rangle } \bS_i(x_t)$.   

Suppose that $b = \ell$, i.e., $x_b = f_i$, and let $y \seteq \bS_i(f_i) = -e_i t_i$. Since $ e_i^\star (y_{a+1})=0 $, Lemma \ref{Lem: Py} and \eqref{Eq: xaxb} say that 
\begin{equation} \label{Eq: [p+1]y_k+1}
\begin{aligned} 
 [p_{\beta_b,\beta_a}+1] y_{t+1} &= \bS_i( x_ax_b - q^{-(\beta_a, \beta_b )} x_b x_a )   =  y_{a+1} y - q_i^{ \langle h_i, \gamma_{a+1} \rangle} y y_{a+1} \allowdisplaybreaks\\
&= \frac{ q_i^{ \langle h_i, \gamma_{a+1} \rangle} }{ q_i^2 (q_i^{-1} - q_i) } e_i'( y_{a+1} )
\end{aligned}
\end{equation}
where the second equality follows from that $ q^{-(\beta_a, \al_i )}  = q^{ ( \al_i,  s_i\beta_a )} = q_i^{ \langle h_i,  \gamma_{a+1} \rangle}$. Applying $\bT_i$ to \eqref{Eq: xaxb}, by Lemma \ref{Lem: x f_i,k+1}, we have 
\begin{align*}
[p_{\beta_b,\beta_a}+1] \bT_i(x_t) &= \bT_i(x_ax_b - q^{-(\beta_a, \beta_b )} x_b x_a) \allowdisplaybreaks\\
&= \cC_i^{\langle h_i, \beta_a  \rangle } \left( y_{a+1} f_{i,1} - q_i^{  \langle h_i, \gamma_{a+1} \rangle} f_{i,1} y_{a+1} \right) \allowdisplaybreaks \\
&= \cC_i^{\langle h_i, \beta_a  \rangle } q_i^{  \langle h_i ,\gamma_{a+1} \rangle - 1}  ( q_i^{-1} - q_i)  e'_i (y_{a+1})  \allowdisplaybreaks\\
&\underset{\eqref{Eq: [p+1]y_k+1}}{=} [p_{\beta_b,\beta_a}+1] \cC_i^{ \langle h_i, \beta_a  \rangle } q_i (q_i^{-1}-q_i)^2 y_{t+1} \allowdisplaybreaks \\
&= [p_{\beta_b,\beta_a}+1] \cC_i^{ \langle h_i, \beta_{t}  \rangle } y_{t+1},
\end{align*}
where the last equality follows from that $\beta_{t} = \beta_a + \al_i$. Thus we have the assertion.
\end{proof}

\begin{lemma}[cf.~{\cite[Corollary 8.4 (b)]{JLO2}}] \label{Lemma: Tw0}
Let $w_\circ$ be the longest element of $\weyl$, and let $\inv$ be the Dynkin diagram automorphism defined by $i \mapsto i^*$, where $\al_{i^*} = -w_\circ(\al_i)$. Then we have $ \bT_{w_\circ} = \dD \circ \inv $, i.e.,
$$
\bT_{w_\circ}(f_{i,p}) = f_{i^*, p+1} \qquad \text{ for any $i\in I$ and $p \in \Z$.}
$$
\end{lemma}
\begin{proof}
It suffices to show that $\bT_{w_\circ}(f_{i,0}) = f_{i^*, 1}$ for any $i\in I$.
Let $w = w_\circ s_i = s_{i^*} w_\circ$ and set $\bfi  = (i_1, \ldots, i_{\ell-1}) \in R(w)$. 
When $\cmC$ is of symmetric type, we understand all roots are short.
We set 
\bna
\item $ \beta^\vee \seteq 2 \beta / (\beta, \beta) $ for a positive root $\beta \in \Phi_+$, 
\item $\Phi_{+,\mathrm{s}}^\vee \seteq \{ \beta^\vee \mid  \beta\text{ : short positive roots} \} $ and $\Phi_{+,\mathrm{l}}^\vee \seteq \{ \beta^\vee \mid  \beta\text{ : long positive roots} \} $, 
\item $2\rho_s^\vee \seteq \sum_{\beta^\vee \in \Phi_{+,s}^\vee} \beta^\vee$, and $2\rho_l^\vee \seteq \sum_{\beta^\vee \in \Phi_{+,l}^\vee} \beta^\vee$.
\ee
Note that $ s_i(\beta^\vee) = (s_i\beta)^\vee $ for any $\beta\in \Phi_+$.

We assume that $\al_i$ is short. The we have $ (2\rho_s^\vee, \al_i)=2$ and $ (2\rho_l^\vee, \al_i)=0$.
Define 
$\upepsilon(i)\seteq 1$ if $\al_i$ is short, $\upepsilon(i)\seteq 0$ otherwise. 
Then we have 
\begin{align*}
\sum_{k=1}^{\ell-1}  \upepsilon(i_k) \langle h_{i_k}, s_{i_{k+1}} \cdots s_{i_{\ell-1}} (\al_i)   \rangle  
& = \sum_{k=1}^{\ell-1} \upepsilon(i_k) ( \al_{i_k}^\vee, s_{i_{k+1}} \cdots s_{i_{\ell-1}} (\al_i)   ) \allowdisplaybreaks \\
&= \sum_{k=1}^{\ell-1} \upepsilon(i_k) ( s_{i_{\ell-1}}  \cdots s_{i_{k+1}}  (\al_{i_k}^\vee),  \al_i    )  \allowdisplaybreaks\\
&= \sum_{ \beta^\vee \in \Phi_{+,s}^\vee \setminus \al_i^\vee} ( \beta^\vee,  \al_i    )  \allowdisplaybreaks\\
&=  ( 2\rho_s^\vee - \al_i^\vee,  \al_i    ) = 0,
\end{align*}
and 
\begin{align*}
\sum_{k=1}^{\ell-1}  (1-\upepsilon(i_k)) \langle h_{i_k}, s_{i_{k+1}} \cdots s_{i_{\ell-1}} (\al_i)   \rangle  
& = \sum_{k=1}^{\ell-1} (1-\upepsilon(i_k)) ( \al_{i_k}^\vee, s_{i_{k+1}} \cdots s_{i_{\ell-1}} (\al_i)   )  \allowdisplaybreaks\\
&= \sum_{k=1}^{\ell-1} (1-\upepsilon(i_k)) ( s_{i_{\ell-1}}  \cdots s_{i_{k+1}}  (\al_{i_k}^\vee),  \al_i    ) \allowdisplaybreaks \\
&= \sum_{ \beta^\vee \in \Phi_{+,l}^\vee } ( \beta^\vee,  \al_i    )  \allowdisplaybreaks\\
&=  ( 2\rho_l^\vee,  \al_i    ) = 0.
\end{align*}
Proposition \ref{Prop: T and S} says that 
$$ 
\bT_{w} (f_{i,0}) = \cC_s^{  \sum_{k=1}^{\ell-1} \upepsilon(i_k) \langle h_{i,k}, s_{i_{k+1}} \cdots s_{i_{\ell-1}} ( \al_i) \rangle }  \cC_l^{ \sum_{k=1}^{\ell-1} (1-\upepsilon(i_k)) \langle h_{i,k}, s_{i_{k+1}} \cdots s_{i_{\ell-1}} ( \al_i) \rangle } f_{i^*, 0} = f_{i^*, 0},
$$
which implies that
$$
\bT_{w_\circ} (f_{i,0}) =  \bT_{i^*} \bT_{w} (f_{i,0}) = f_{i^*,1}.
$$

In the case where $\al_i$ is long, it can be proved in the same manner as above. 
\end{proof}

\subsection{Bilinear form on $\qbA{}$}
The bilinear forms $(\ ,\ )_K$ and $(\ ,\ )_L$ on $\calU_q^-(\g)$ given in Section \ref{Subsection: bf and qba} are also defined on $\qbA{[k]}$ via the identification $\qbA{[k]} \simeq \calU_q^-(\g)$ given in Lemma \ref{Lem: id A and U}.

For any $\beta = \sum_{i\in I} n_i \al_i \in \rlQ$, we set 
$$
\cC^\beta \seteq \prod_{i\in I} \cC_i^{n_i}.
$$

\begin{definition} \label{def: ( )}
Let $a,b \in \Z$ with $a \le b$.
For homogeneous elements $x=x_{b} x_{b-1} \cdots x_{a}$ and $ y = y_{b} y_{b-1} \cdots y_{a} $ in $ \qbA{[a,b]}$, 
we define a bilinear form $\llf \ , \ \rrf$ on $\qbA{[a,b]}$ by
\begin{align*} 
\pair{x,y} \seteq \prod_{k\in [a,b]}  \cC^{  \abs{\wt(x_k)} + \abs{\wt(y_k)} } (x_k, y_k)_L  \in \bR,  
\end{align*}
where $\abs{\beta}$ is defined in \eqref{eq: height and obe}.
\end{definition}
Note that $\abs{\wt(x_k)} = (-1)^{k+1} \wt(x_k)$ for $k\in [a,b]$. It is easy to see that the bilinear form $\pair{ \ , \ }$ can be extended to the whole algebra $\qbA{}$.

\begin{lemma} \label{Lem: properties of (,)} \
\bnum 
\item \label{it: pairing1} The bilinear form $\pair{ \ , \ }$ is well-defined and it is symmetric and non-degenerate.
\item \label{it: pairing2} Let $x=x_{b} x_{b-1} \cdots x_{a}$ and $ y = y_{b} y_{b-1} \cdots y_{a} $ be homogeneous elements in $ \qbA{[a,b]}$.
\bna
\item We have  
$$ 
\pair{x,y}  = \prod_{k\in [a,b]}\pair{x_k,y_k}. 
$$
In particular, we have $ \pair{x,y} = 0$ unless $\wt(x_k) = \wt(y_k)$ for all $k\in [a,b]$. 

\item  If we write $ \abs{\wt(x_t)} = \sum_{i\in I}  n_{i,t} \al_i \in \rlQ_+$ for $t\in [a,b]$, then 
$$
\pair{x,y} = \prod_{t\in [a,b]}  q^{N_t} \cC^{\abs{\wt(x_t)}} (x_t, y_t)_K, 
$$
where $N_t = \displaystyle \frac{1}{4} \sum_{i\in I} (\al_i, \al_i) n_{i,t}$.
\ee
\item \label{it: pairing3} For any $x,y \in \qbA{[a,b]}$, we have 
$$
\pair{\dD x, \dD y} = \pair{x,y} 
 \quad \text{and} \quad \pair{*x,*y} = \pair{x,y}. 
$$
\ee
\end{lemma}
\begin{proof}
\eqref{it: pairing1} Since $(\ , \ )_L$ is a non-degenerate symmetric bilinear form on $\calU_q^-(\g)$, the assertion follows from by Lemma \ref{Lem: qA decomp}. 

\noindent
\eqref{it: pairing2} Since $(u,v)_L = 0$ for any homogeneous elements $ u,v \in \calU_q^-(\g) $ with $\wt(u) \ne \wt(v)$, the assertion (a) follows from Definition \ref{def: ( )}.
By \eqref{eq: Kform Lform}, we have the assertion (b).

\noindent
\eqref{it: pairing3} It follows from Definition \ref{def: ( )} and \cite[Lemma 1.2.8]{LusztigBook}.
\end{proof}

\vskip 2em

\section{PBW bases: locally reduced sequences} \label{Sec: PBW locally reduced} 

We shall construct and investigate a PBW bases of $\qbA{}$. We first define PBW vectors and PBW monomials in a formal way, and investigate their properties in this section and Section \ref{Section: general PBW}. 

\begin{definition} Let $[u,v]$ be an interval of $\Z$. A pair $ \pd = (\ii, \xi)$ of $\ii \in I^{[u,v]}$ and $\xi \in \Z$ is called a \emph{PBW datum}. 
\end{definition}

Let  $ \pd = (\ii, \xi)$ be a PBW datum of $\ii = (i_k)_{k\in [u,v]} \in I^{[u,v]}$ and $\xi \in \Z$. For any $ k \in [u,v]$, we define 
\begin{align} \label{Eq: PBW vectors} 
\rF^{\pd}_{k} \seteq \bT_{i_u} \bT_{i_{u+1}} \cdots \bT_{i_{k-1}} (f_{i_k,{\xi}})  \qtq 
\rFv^{ \pd}_{k} \seteq \bT_{i_u}^{-1} \bT_{i_{u+1}}^{-1} \cdots \bT_{i_{k-1}}^{-1} (f_{i_{k},{\xi}}). 
\end{align}
We call $ \rF^{ \pd }_k$ and $ \rFv^{ \pd }_k$ the $k$-th \emph{PBW root vectors} associated with the PBW-datum $\pd$. Lemma \ref{Lem: Ti and others}~\eqref{it: bar T} says that the PBW root vectors $\rF^{\pd}_{k}$ and $\rFv^{\pd}_{k}$ are bar-invariant.
We frequently write $\rF_{k}$ and $\rFv_{k}$ instead of $\rF_{k}^\pd$ and $\rFv^{ \pd}_{k}$ if no confusion arises.

For any $\bsd = (d_u, d_{u+1}, \ldots d_v) \in \Z_{\ge 0}^{\oplus [u,v]}$, we define 
\begin{equation} \label{Eq: PBW monomials}
\begin{aligned}
\rF_{\pd}(\bsd) &\seteq \Opd^{\longrightarrow}_{k\in [u,v]}  \rF^{d_k}_k  =    \rF^{d_v}_v    \rF^{d_{v-1}}_{v-1}   
\cdots  \rF^{d_u}_u, \\
\rFv_{\pd}(\bsd) &\seteq \Opd^{\longleftarrow}_{k\in [u,v]}  \rFv^{d_k}_k  =    \rFv^{d_u}_u    \rFv^{d_{u+1}}_{u+1}   \cdots  \rFv^{d_v}_v.
\end{aligned}
\end{equation}
The monomials $\rF_{\pd}(\bsd)$, $\rFv_{\pd}(\bsd)$ are called  \emph{PBW monomials} associated with  $\pd$ and $\bsd$. 
\begin{remark} 
We do \emph{not} use divided powers in the definition \eqref{Eq: PBW monomials} of PBW monomials.  
\end{remark}

\subsection{Case for $\qbA{[a,b]}$} 
Let $a \in \Z$ and $b\in \Z \cup \{\infty \}$ with $a \le b$. 
We set $\ell\seteq \ell(w_\circ)$ and  
$$  
N \seteq \ell \times (b-a+1).
$$
We choose a locally reduced sequence $\bfi^\Diamond = (i_k)_{k\in \Z}   \in I^\Z$; 
i.e.,
\begin{align} \label{Eq: longest case}
\text{$ (i_k, i_{k+1}, \ldots, i_{k+\ell-1})$ is a reduced expression of $w_\circ$ for any $k \in \Z$.} 
\end{align}
It is easy to see that $i_{k+\ell} = i_k^*$ for any $k \in \Z$. We set 
\begin{align} \label{Eq: i circ}
\bfi \seteq (i_1, i_2, \ldots, i_N) \in I^{[1,N]} \quad \text{ and }\quad 
\bfic \seteq (i_1, i_2, \ldots, i_\ell) \in R(w_\circ).
\end{align}
 We define the PBW datum 
$$
\pd\seteq ( \bfi, a ).
$$
We now write the PBW root vectors $\rF_{k}$ instead of $\rF_{k}^\pd$ for $k\in [1,N]$ if no confusion arises.  
 
\begin{lemma} \label{Lem: Fk+ell=Dfk} \
\bnum
\item \label{it: Dlk1} For any $k\in [1, N-\ell]$, we have 
$$
\rF_{k+\ell} = \dD( \rF_k ).
$$
\item \label{it: Dlk2} We consider the PBW vectors $\calF_r$ associated with the reduced expression $\bfic$ given in \eqref{eq: pbw vector U}
for $1 \le r \le \ell$. Let $k\in [1,N]$ and write $k = q \ell + r$ for some $q,r\in \Z_{\ge0}$ with $r \in [1, \ell]$. 
Under the identification $\qbA{[a]} \simeq \calU_q^-(\g)$ given in {\rm Lemma \ref{Lem: id A and U}}, we have
$$
\rF_k = \dD^q ( \calF_r ) \qquad \text{$($up to a product of $\cC_i$'s$)$.}
$$
\ee
\end{lemma}
\begin{proof}
\eqref{it: Dlk1} 
Since $ i_{k+\ell} = i_k^* $,  
it follows from Lemma \ref{Lem: Ti and others}, Lemma \ref{Lemma: Tw0} and \eqref{Eq: longest case} that 
\begin{align*}
\rF_{k+\ell} &= \bT_{i_1} \cdots \bT_{i_{k-1}} \left( \bT_{i_k} \cdots \bT_{i_{k+\ell-1}} \right) (f_{i_{k+\ell}, a})  \allowdisplaybreaks\\
&= \bT_{i_1} \cdots \bT_{i_{k-1}} \bT_{w_\circ}  (f_{i_{k}^*, a})\allowdisplaybreaks \\
&= \bT_{i_1} \cdots \bT_{i_{k-1}} \dD  (f_{i_{k}, a}) \allowdisplaybreaks\\
&=  \dD  ( \rF_k). 
\end{align*}

\noindent
\eqref{it: Dlk2} It follows from~\eqref{it: Dlk1} and Proposition \ref{Prop: T and S}.
\end{proof}

Thanks to Lemma \ref{Lem: Fk+ell=Dfk}, every PBW root vector $\rF_{k}$ ($k\in [1,N]$) is contained in $\qbA{[a,b]}$.
Recall the bilinear form $\pair{ \ , \ }$ defined in Definition \ref{def: ( )}.

\begin{theorem} \label{Thm: orthonormal_reduced}
Let $\sfP_{\pd} \seteq \{  \rF_{\pd}(\bsu) \mid \bsu \in \Z_{\ge 0}^{\oplus N} \} $.
\bnum
\item \label{it: reduced ortho1}
The set $\sfP_{\pd}$ is a $\bR$-linear orthogonal basis of $\qbA{[a,b]}$ with respect to $\pair{ \ , \ }$.
\item  \label{it: reduced ortho2}
For any $k\in [1,N]$ and $u \in \Z_{>0}$, we have  
$$
\pair{  \rF^{u}_k ,  \rF^{u}_k } = q_{i_k}^{-u(u-1)/2} ( q_{i_k}^{-1} - q_{i_k})^u [u]_{i_k}!.
$$
\item  \label{it: reduced ortho3} For any $\bsu = (u_1, \ldots, u_N), \bsu' = (u_1', \ldots, u_N' ) \in \Z_{\ge0}^{\oplus N}$, we have 
$$
\pair{\rF_{\pd}(\bsu), \rF_{\pd}(\bsu')}
 = \prod_{k=1}^N \delta_{u_k, u_k'} q_{i_k}^{-u_{k}(u_{k}-1)/2} (q_{i_k}^{-1}-q_{i_k})^{u_k}  [u_{k}]_{i_k}!.
$$
\ee
\end{theorem}
\begin{proof}
\eqref{it: reduced ortho1}
For any $t\in [0,b-a]$ and $r \in [1,\ell]$, we define $ \rF_{t, r} \seteq \rF^{\pd}_{t \ell + r} $, and set 
$$ 
\rF_{t}(\bsc) \seteq \Opd^{\longrightarrow}_{r \in [1,\ell]}  \rF_{t, r}^{c_r},
$$ for any $\bsc = (c_1, \ldots, c_\ell) \in \Z_{\ge0}^{\oplus \ell}$.
Thanks to Lemma \ref{Lem: Fk+ell=Dfk}, under the identification $\qbA{[a+t]} \simeq \calU_q^-(\g)$ given in Lemma \ref{Lem: id A and U}, $\rF_{t}(\bsc)$ is equal to the PBW monomial in $\calU_q^-(\g)$ associated with $\bsc$ and $\bfic$ up to a non-zero constant. Since the  PBW monomials form an orthogonal basis of $\calU_q^-(\g)$ with respect to $(\ , \ )_L$ (see \cite[Proposition 38.2.3]{LusztigBook} and see also \cite[Proposition 4.22]{Kimura12}), the set 
$$
\left\{ \rF_{t}(\bsc) \mid \bsc \in \Z_{\ge0}^{\oplus \ell}  \right\}
$$
is an orthogonal basis of $\qbA{[a+t]}$ with respect to $\pair{\ , \ }$. 
By Lemma \ref{Lem: qA decomp} and Definition \ref{def: ( )} of the bilinear form  $\pair{\ , \ }$, we have the assertion. 

\smallskip

\noindent
\eqref{it: reduced ortho2} 
By Definition \ref{def: ( )}, Lemma \ref{Lem: properties of (,)}~\eqref{it: pairing3} and Lemma \ref{Lem: Fk+ell=Dfk}, we may assume that $k\in [1, \ell]$.
Let $x,y$ be homogeneous elements in $\qbA{[a]}$ such that $\bT_i(x), \bT_i(y) \in \qbA{[a]}$. We set $\beta \seteq \wt(x)$ and $\gamma \seteq \wt(y)$. Note that 
\begin{align*}
\abs{\beta} &=  (-1)^{a+1}\beta \quad \text{ and } \quad  \abs{s_i(\beta)} = (-1)^{a+1}s_i(\beta), \\
\abs{\gamma} &=  (-1)^{a+1}\gamma \quad \text{ and } \quad  \abs{s_i(\gamma)} = (-1)^{a+1}s_i(\gamma).
\end{align*}
Since $(\ , \ )_L$ is invariant under the action $\bS_i$ (see \cite[Prospoition 38.2.1]{LusztigBook}), 
Proposition \ref{Prop: T and S} and Definition \ref{def: ( )} tell us that 
\begin{align*}
\pair{\bT_i(x), \bT_i(y)} &=  \cC^{ (-1)^{a+1} ( s_i(\beta) + s_i(\gamma))} (\bT_i(x), \bT_i(y))_L \allowdisplaybreaks\\ 
&=  \cC^{(-1)^{a+1}(\beta  +  \gamma)} ( \cC_i^{-\langle h_i, (-1)^{a+1}\beta \rangle} \bT_i(x), \cC_i^{-\langle h_i,  (-1)^{a+1}\gamma  \rangle} \bT_i(y))_L \allowdisplaybreaks\\
&=  \cC^{\abs{\beta}  + \abs{ \gamma}} (\bS_i (x), \bS_i(y))_L \allowdisplaybreaks\\
&=  \cC^{\abs{\beta}  + \abs{ \gamma}} (x, y)_L \allowdisplaybreaks\\
&=  \pair{x, y}.
\end{align*}
By \cite[Lemma 1.4.4]{LusztigBook}, for any $u\in \Z_{>0}$, we have 
\begin{align*}
\pair{  \rF^{u}_k,  \rF^{u}_k } &= \pair{ \bT_{i_1} \cdots \bT_{i_{k-1}} (f_{i_k,a}^u),  \bT_{i_1} \cdots \bT_{i_{k-1}} (f_{i_k,a}^u) } \allowdisplaybreaks\\
&= \pair{f_{i_k,a}^u, f_{i_k, a}^u} \allowdisplaybreaks\\
&= \cC_{i_k}^{2u} (f_{i_k,a}^u, f_{i_k, a}^u)_L \allowdisplaybreaks\\
&=  q_i^{-u(u-1)/2} (q_{i_k}^{-1} - q_{i_k})^u [u]_{i_k}!.
\end{align*}

\smallskip

\noindent
\eqref{it: reduced ortho3}  By Definition \ref{def: ( )} and \cite[Proposition 38.2.3]{LusztigBook}
(see also \cite[Proposition 4.22]{Kimura12}), we have 
$$
\pair{\rF_{\pd}(\bsu), \rF_{\pd}(\bsu')} = \Opd_{k\in [1,N]} \delta_{u_k, u_k'} \pair{ \rF^{u_k}_k, \rF^{u_k}_k},
$$
which implies the assertion by~\eqref{it: reduced ortho2}.
\end{proof}

The following is an analogue of the \emph{Levendorskii-Soibelman formula} (shortly LS-formula) (see \cite[Proposition 5.2.2]{LS91}).  
\begin{lemma} \label{Lem: pbw comm} \
Let $k, t \in [1,N] $ with $ k < t$. Then we have 
$$
\rF_k \rF_t - q^{ - ( \wt(\rF_k), \wt (\rF_t) ) } \rF_t \rF_k = \sum_{\bsc= (c_{k+1}, \ldots, c_{t-2}, c_{t-1}) \in \Z_{\ge 0}^{\oplus (k,t)}} A_\bsc   \rF^{c_{t-1}}_{t-1}  
\rF^{c_{t-2}}_{t-2}   \cdots  \rF^{c_{k+1}}_{k+1},
$$
for some $ A_\bsc \in \bR$. In particular, if $ t > k+\ell$, then 
$$
\rF_k  \rF_t = q^{ - ( \wt(\rF_k), \wt (\rF_t) ) } \rF_t \rF_k.
$$
\end{lemma}
\begin{proof}
Since $\bT_i$ is an automorphism, we may assume that $k=1$, i.e., $\rF_k= f_{i_1, a}$.

\noindent
$\bullet$ If $ t \le \ell$, then it follows from Proposition \ref{Prop: T and S} and the LS-formula (\cite[Proposition 5.2.2]{LS91}). 

\noindent
$\bullet$ If $t = \ell+1$, then it follows from Definition \ref{Def: extended qg} (c). 

\noindent
$\bullet$ Suppose that $ t > \ell+1$.
We write $ t = v \ell + r $ for some $v\in \Z_{\ge 1}$ and $ r \in [0,\ell-1]$. By Lemma \ref{Lemma: Tw0} and \eqref{Eq: longest case}, we have 
$$
\rF_t = \bT_{i_1} \bT_{i_2} \cdots \bT_{i_{r}} \bT_{w_\circ}^v ( f_{i_{t}, a} )
= \bT_{i_1} \bT_{i_2} \cdots \bT_{i_{r}} ( f_{i_{t+ \ell v}, a+v} ). 
$$
Since $v\ge 1$, by the definition $\bT_i$ and Definition \ref{Def: extended qg} (b), one can easily see that $f_{i_1, a-1}$ commutes with $ \bT_{i_1}^{-1} \rF_t = \bT_{i_2} \cdots \bT_{i_{r}} ( f_{i_{t+ \ell v}, a+v} )$ up to a power of $q$, i.e., 
$$
 (\bT_{i_1}^{-1} \rF_t) f_{i_1, a-1} =  
 q^{-( \wt( f_{i_1, a-1} ), \wt( \bT_{i_1}^{-1} \rF^{\bfi}_t ) )} f_{i_1, a-1}   (\bT_{i_1}^{-1} \rF_t).
$$
Applying $\bT_{i_1}$ to the above, we have the assertion. 
\end{proof}

The following proposition says that the bilinear form  is invariant under the action $\bT_i$.

\begin{proposition} \label{Prop: invariant form}
Let $i\in I$. For any $x,y \in \qbA{} $, we have   
$$
\pair{\bT_i(x), \bT_i(y)} = \pair{x,y}.
$$
\end{proposition}
\begin{proof}
Since $\pair{\ , \ }$ is invariant under the action $\dD$, it suffices to show that $\pair{\bT_i(x), \bT_i(y)} = \pair{x,y}$ for any $x, y \in \qbA{\ge 0}$. 

We choose a locally reduced sequence $\bfi^{\Diamond} = (i_k)_{k\in \Z}$ satisfying \eqref{Eq: longest case} and $i_0 = i$, 
and set another $\bfj^\Diamond \seteq (j_k)_{k\in \Z}$ such that $ j_k = i_{k-1}$ for any $k\in \Z$. We set 
\begin{align*} 
\bfi \seteq (i_1, i_2, i_3, \ldots) \in I^{[1, \infty]}  \quad \text{ and }\quad 
\bfj \seteq (j_1, j_2, j_3, \ldots) \in I^{[1, \infty]} 
\end{align*} 
and define PBW data $\pd\seteq ( \bfi, 0 )$ and $\pdh\seteq ( \bfj, 0 )$.
Note that 
$\rF^{\pdh}_{k+1} = \bT_i (\rF^{\pd}_k) $ for any $k\ge 1$ and $\rF^{\pdh}_{1}= f_{i,0}$.

Let $\bsu : =(u_k)_{k\ge 1}, \bsu' : =(u_k')_{k\ge 1} \in \Z_{\ge0}^{\oplus \Z_{ \ge1}}$ and set $\bsv : =(v_k)_{k\ge 1}$ and $\bsv' : =(v_k')_{k\ge 1}$ such that 
$$
\text{$ v_1=v_1'=0$, $v_{k} = u_{k-1}$ and $v_{k}' = u_{k-1}'$ for any $k>1$.} 
$$
Since the PBW monomials $\rF_{\pd}(\bfd)$ form an orthogonal basis of $\qbA{\ge 0}$ by Theorem \ref{Thm: orthonormal_reduced},  we may assume that $x = \rF_{\pd}(\bsu)$ and $y = \rF_{\pd}(\bsu')$. 
By construction, we have $ \bT_i(x) =  \rF_{\pdh}(\bsv)$ and $\bT_i(y) = \rF_{\pdh}(\bsv')$,
 and Theorem \ref{Thm: orthonormal_reduced} tells us that 
\begin{align*}
\pair{ x, y } & = \prod_{k=1}^\infty \delta_{u_k, u_k'} q_{i_k}^{-u_{k}(u_{k}-1)/2} (q_{i_k}^{-1}-q_{i_k})^{u_k}  [u_{k}]_{i_k}! \allowdisplaybreaks\\
& = \prod_{k=2}^\infty \delta_{v_k, v_k'} q_{j_k}^{-v_{k}(v_{k}-1)/2} (q_{j_k}^{-1}-q_{j_k})^{v_k}  [v_{k}]_{j_k}! \allowdisplaybreaks\\
&= \pair{ \bT_i(x), \bT_i(y) },
\end{align*}
where the last equality follows from that $v_1=v_1'=0$. 
\end{proof}

\subsection{Case for $\qbA{}$} 

We keep the notations appeared in the previous subsection.
Let $\bfi^\Diamond =(i_k)_{k\in \Z}$ be a sequence satisfying \eqref{Eq: longest case}. 
We  set  $j_k = i_{-k}$ for $k\in [0,\infty]$ and  
\begin{align*} 
\bfi^- \seteq (j_0, j_{1}, j_2, \ldots) \in I^{ [0, \infty] }  \quad \text{ and }\quad 
\bfi \seteq (i_1, i_2, i_3, \ldots) \in I^{[1, \infty]}.
\end{align*}
The two PBW root vectors in~\eqref{Eq: PBW vectors} are related as follows.

\begin{lemma} \label{lem: F anf vF}
Let $a,b \in \Z$ with $a \le b$ and set $N\seteq \ell \times (b-a+1)$. 
Then we have 
$$
\rF^{(\bfi,a)}_{k} = \rFv^{(\bfi^-,b)}_{N-k}  \quad \text{for any $k \in [1, N]$.}
$$
\end{lemma}
\begin{proof}
We set $ \mathcal{S}\seteq \bT_{i_{k-N+1}} \cdots \bT_{i_{-1}}  \bT_{i_{0}} $ and $ \bT\seteq \bT_{i_1} \bT_{i_2} \cdots \bT_{i_{k-1}} $. 
Since $\mathcal{S} \bT =  \bT_{ s_{i_{k-N}} w_\circ} (\bT_{w_\circ})^{b-a}  $, we have 
$$ 
\mathcal{S} \bT (f_{i_k, a}) =  \bT_{i_{k-N}}^{-1} (\bT_{w_\circ})^{b-a+1} (f_{i_{k}, a})
= \bT_{i_{k-N}}^{-1} (f_{i_{k-N}, b+1}) = f_{i_{k-N}, b}
$$
by Lemma \ref{Lemma: Tw0}. Thus we have $ \bT (f_{i_k, a}) = \mathcal{S}^{-1} (f_{i_{k-N}, b})$, which gives the assertion. 
\end{proof}

We define the PBW datum $\pdh\seteq ( \bfi^-, -1 )$ and $\pd\seteq ( \bfi, 0)$, and for $k\in [1,\infty]$ and $t\in [-\infty, 0]$ set 
$$
\rF_k\seteq \rF^{\pd}_{k} \quad \text{ and } \quad  \rF^\star_{t}\seteq \rFv^{\pdh}_{-t}.
$$

Let  $\bst = (t_{k})_{k\in \Z} \in \Z_{\ge0}^{\oplus \Z}$ and write
$\bst^- = (t_{k} )_{k \le 0} \in \Z_{\ge0}^{\oplus \Z_{\le 0}}$ and $  \bst^+  = (t_{k} )_{k > 0} \in \Z_{\ge0}^{\oplus \Z_{>0}}$ so that $\bst = \bst^- * \bst^+ $. Here $A*B$ is the concatenation of $A$ and $B$.
We define 
$$
\rFt_{\bfi^\Diamond}(\bst) \seteq \rF_{\pd}(\bst^+) \cdot  \rF_{\pdh}^\star(\bst^-) , 
$$
where 
\begin{align*}
\rF_{\pdh}^\star (\bst^-) \seteq \Opd^{\longrightarrow}_{k\in [-\infty,0]}  (\rF^{\star}_k)^{u_k} = 
(\rF^{\star}_0)^{u_0} \cdot (\rF^{\star}_{-1})^{u_{-1}} 
\cdot (\rF^{\star}_{-2})^{u_{-2}} \cdots 
\end{align*}

Let $a \in \Z_{<0}$ and $b \in \Z_{>0}$ and set 
$$ 
\bfZ_{[a,b]}\seteq \{ \bst = (t_{k})_{k\in \Z} \in \Z_{\ge0}^{\oplus \Z} \mid t_k=0 \text{ unless $ a \ell < k \le (b+1)\ell $} \}.
$$
It follows by applying Lemma \ref{Lem: qA decomp}, Lemma \ref{lem: F anf vF} and Theorem \ref{Thm: orthonormal_reduced} to $ \qbA{[a,b]} \simeq  \qbA{[0,b]} \otimes_\bR \qbA{[a,0)}$ that the set
$$
\sfP_{\bfi^\Diamond;[a,b]} \seteq \{ \rFt_{\bfi^\Diamond}(\bst)  \mid \bst \in \bfZ_{[a,b]} \}
$$
forms an orthogonal basis of $\qbA{[a,b]}$ satisfying the same properties of Theorem \ref{Thm: orthonormal_reduced}.
Considering the case when $a \rightarrow -\infty$ and $b \rightarrow \infty$ respectively, we obtain the following.

\begin{corollary} \label{Cor: orthonormal_reduced_Aq}
Let $a \in \Z_{<0} \cup \{ -\infty\} $ and $  b\in \Z_{>0} \cup \{ \infty \}$.
\bnum
\item The set $\sfP_{\bfi^\Diamond;[a,b]}$ is a $\bR$-linear orthogonal basis of $\qbA{[a,b]}$ with respect to $\pair{\ , \ }$.
\item For any $\bst = (t_{k}), \bst' = (t_{k}') \in \Z_{\ge0}^{\oplus (a\ell, (b+1)\ell]}$, we have 
$$
\pair{\rFt_{\bfi^\Diamond}(\bst), \rFt_{\bfi^\Diamond}(\bst')} = \prod_{k= a\ell+1}^{(b+1)\ell} \delta_{t_k, t_k'} q_{i_k}^{-t_{k}(t_{k}-1)/2} (q_{i_k}^{-1}-q_{i_k})^{t_k}  [t_{k}]_{i_k}!.
$$
\ee
\end{corollary}

Note that the monomials $\rF_{\pd}(\bst^+)$ (resp.\ $\rF_{\pdh}^\star(\bst^-)$) form a $\bR$-linear orthogonal basis of $\qbA{\ge0}$ (resp.\ $\qbA{<0}$) with respect to $\pair{\ , \ }$.
For any $i\in I$ and $a\in \Z$, we set
\begin{align*}
(\qbA{\ge a})_i &\seteq \Span_\bR \{ xy \mid x \in \qbA{>a}, \ y \in \qbA{[a]} \text{ with $\es_i(y)=0$}  \}, \allowdisplaybreaks\\
{_i}(\qbA{\le a}) &\seteq \Span_\bR \{ xy \mid y \in \qbA{<a}, \ x \in \qbA{[a]} \text{ with $e'_i(x)=0$}  \}.
\end{align*}
\begin{lemma} \label{Lem: codomain of T_i}
Let $i\in I$ and $a\in \Z$.
\bnum
\item $(\qbA{\ge a})_i$ and $_{i}(\qbA{\le a})$ are subalgebras of $\qbA{}$. 
\item For any $u \in \qbA{\ge a} $,  we have $\bT_i(u)\in (\qbA{\ge a})_i.$
\item For any $u \in \qbA{\le a} $,  we have $\bT^{-1}_i(u)\in {_{i}}(\qbA{\le a}).$
\ee
\end{lemma}
\begin{proof}
We shall focus on the case for $(\qbA{\ge a})_i$ since the other case for ${ _i}(\qbA{\le a})$ can be proved in the same manner. Without loss of generality, we assume that $a=0$.

We choose a locally reduced sequence $\bfi = (i_k)_{k\in \Z_{\ge 1}}   \in  I^{ [1, \infty]}$ such that $i_1 = i$. Set $\pd\seteq (\bfi, 0) $.
Then $\rF^{\pd}_1 =f_{i, 0}$ and 
$ \es_i( \rF^{\pd}_k ) = 0 $ for any $ k=2, \ldots, \ell$ by Proposition \ref{Prop: T and S}.
Thus all monomials of the PBW vectors $ \rF^{\pd}_k $ ($k \in \Z_{>1}$) form  an orthogonal basis of $(\qbA{\ge 0})_i$ by Corollary \ref{Cor: orthonormal_reduced_Aq}. Lemma \ref{Lem: pbw comm} says that $(\qbA{\ge 0})_i$ is a subalgebra. 

By the definition of $\bT_i$, it is obvious that $ \bT_i(u) \in \qbA{\ge a} $ for any $a\in \Z$ and $u \in \qbA{\ge a} $.
Let $x\in \qbA{\ge 1} $, $y  \in \qbA{[0]}$ with $e'_i(y)=0$. By Theorem \ref{Prop: T and S}, we have 
\begin{align*}
\bT_i (x f_{i,0}^m y) = \bT_i (x) \bT_i( f_{i,0}^m) \bT_i( y) = 
\left(\bT_i (x) f_{i,1}^m \right) \bT_i(y) \in (\qbA{\ge 0})_i
\end{align*}
for any $m\in \Z_{\ge0}$.
Since such elements $x f_{i,0}^m y$ spans $\qbA{\ge 0}$, we have $ \bT_i(u) \in (\qbA{\ge 0})_i $ for any $u \in \qbA{\ge 0} $.
\end{proof}

\begin{example} \label{Ex: example1}
For homogeneous elements $x, y \in  \qbA{}$, we define 
$$
[x,y]_q\seteq xy - q^{-( \wt(x), \wt(y)) } yx.
$$
Note that $\bT_i( [x,y]_q) = [\bT_i(x),  \bT_i(y)]_q$ 
because $(\beta, \gamma) = ( s_i(\be), s_i(\gamma))$ for any $\beta, \gamma \in \rl$.
One can show that  
\bnum
\item If $[x,z]_q=0$, then  $[x, [y,z]_q]_q = [[x,y]_q,z]_q$, 
\item If $[x,y]_q=0$, then  $[x, [y,z]_q]_q = q^{-( \alpha, \beta)} y [x,z]_q  - q^{-( \beta, \gamma)}  [x,z]_qy $, 
\item \label{it: 4th computation} If $[y,z]_q=0$, then $[ [x,y]_q, z ]_q = q^{-(\beta, \gamma)} [x,z]_q y - q^{-(\al, \beta)} y [x,z]_q$,
\ee
for any homogeneous elements $x,y,z$.

Let $\cmC$ be of type $A_2$ and let $a=0$ and $b=1$. Note that $\ell = \ell(w_\circ)=3$. We consider the locally reduced sequence  
$$
\bfi\seteq (1,2,1,2,1,2) \in I^{[1,6]}
$$
and set $\pd\seteq (\bfi, 0)$. We then have $\rF^{\pd}_{1} = f_{1,0}$, $\rF^{\pd}_{2} = \bT_1(f_{2,0})=\cC^{-1}[f_{1,0}, f_{2,0}]_q$ and 
\begin{align*}
\rF^{\pd}_{3} &= \bT_1 \bT_2 ( f_{1,0} ) = \bT_1 ( \kappa^{-1} [f_{2,0}, f_{1,0}]_q )
=  \kappa^{-1} [ \bT_1(f_{2,0}), \bT_1(f_{1,0})]_q \allowdisplaybreaks\\
&= \kappa^{-2} [ [f_{1,0}, f_{2,0}]_q, f_{1,1}]_q \allowdisplaybreaks\\
& = \kappa^{-2} ( q^{-1}[f_{1,0}, f_{1,1}]_qf_{2,0} - q f_{2,0}[f_{1,0}, f_{1,1}]  ) \allowdisplaybreaks\\
&= f_{2,0},
\end{align*}
where the fifth equality follows from~\eqref{it: 4th computation} and $[f_{2,0}, f_{1,1}]_q=0$, and the last equality follows from $[f_{1,0}, f_{1,1}]_q=1-q^2$. By the same manner, we have
\begin{equation*}
\begin{aligned}
 \rF^{\pd}_{4} = f_{1,1}, \quad \rF^{\pd}_{5} = \cC^{-1}[f_{1,1}, f_{2,1}]_q, \quad \rF^{\pd}_{6} = f_{2,1}.
\end{aligned}
\end{equation*}
It is easy to see that $\rF^{\pd}_{1}$, $\rF^{\pd}_{2}$, \ldots, $\rF^{\pd}_{6}$  are bar-invariant. Theorem \ref{Thm: orthonormal_reduced} says that $\rF^{\pd}_{1}$, $\rF^{\pd}_{2}$, \ldots, $\rF^{\pd}_{6}$ generate the algebra $\qbA{[0,1]}$ and the PBW monomials $\rF_{\pd}(\bsu)$ form an orthogonal basis of $\qbA{[0,1]}$.

\end{example}

\vskip 2em

\section{Adjoint operators $\Ep_{i,k}$ and $\Es_{i,k}$} \label{Sec: adjoint operators}
In this section, we define and investigate the operators $\Ep_{i,k}$ and $\Es_{i,k}$ on the bosonic extension $\qbA{}$ which are adjoint 
to the left and right multiplications of $f_{i,k}$ with respect to the bilinear form $\pair{ \cdot, \cdot }$ on $\qbA{}$. 
These operators $\Ep_{i,k}$ and $\Es_{i,k}$ can be viewed as a natural extension of the derivations $e'_{i}$ and $\es_{i}$ on $\Um$ (see~\eqref{Eq: ei ei*}).

\begin{lemma} \label{Lem: fip fip+1}
For any $(i,p)\in \hI$ and $m, n\in \Z_{\ge0}$, we have 
\begin{align*}
f_{i,p} f_{i, p+1} ^m f_{i, p} ^n  &= q_i^{2m} f_{i, p+1}^m f_{i,p}^{n+1} + (1-q_i^{2m}) f_{i, p+1}^{m-1} f_{i, p} ^n, \\
 f_{i, p} ^n f_{i, p-1}^m f_{i,p}  &= q_i^{2m} f_{i, p}^{n+1} f_{i,p-1}^{m} + (1-q_i^{2m}) f_{i, p}^{n} f_{i, p-1}^{m-1}.
\end{align*}
\end{lemma}
\begin{proof}
We may assume that $n=0$. We first focus on the first identity. If $m=1$, then it follows from the defining relations in Definition \ref{Def: extended qg}. 
For any $m>1$, we have 
\begin{align*}
f_{i,p} f_{i, p+1} ^m &= ( q_i^2 f_{i, p+1} f_{i,p} + (1-q_i^2) ) f_{i, p+1}^{m-1} \\ 
&=  q_i^2 f_{i, p+1} f_{i,p} f_{i, p+1}^{m-1} + (1-q_i^2) f_{i, p+1}^{m-1} \\
&= q_i^{2m} f_{i, p+1}^m f_{i,p} + (1-q_i^{2m}) f_{i, p+1}^{m-1},
\end{align*}
where the third equality follows from the induction hypothesis.

The second identity can be proved in the same manner as above.
\end{proof}

Recall that the operators $e_i'$ and $\es_i$ act on $\qbA{[k]}$ for any $k\in \Z$ (see Remark \ref{Rmk: Aq[k] and Aq[k+1]}).

\begin{lemma} \label{Lem: commute for f_ik u}
Let $ i\in I $ and $ k\in \Z$.
\bnum
\item \label{it: comm 1} Let $u \in \qbA{[k+1]}$ be a homogeneous element with $\es_i(u)=0$. Then we have 
$$
f_{i,k} u = q^{- (\wt(f_{i,k}), \wt(u))} u f_{i,k}.
$$
\item \label{it: comm 2} Let $v \in \qbA{[k-1]}$ be a homogeneous element with $e'_i(v)=0$. Then we have 
$$
 v f_{i,k} = q^{- (\wt(f_{i,k}), \wt(v))} f_{i,k} v .
$$
\ee
\end{lemma}
\begin{proof}

\eqref{it: comm 1} 
Let $N\seteq 2\ell$ and let $\bfi = (i_k)_{k\in [1,N]} \in I^{[1,N]}$ be a locally reduced sequence with $i_1 = i$. We set $\pd\seteq (\bfi, k)$ and consider the PBW root vectors 
$\rF_t \seteq \rF^{ \pd}_t$ of $\qbA{[k, k+1]}$ for $t\in [1, N]$ (see \eqref{Eq: PBW vectors}). Note that $\rF^{\bfi}_1 = f_{i,k}$ and $\rF^{\bfi}_{\ell+1} = f_{i,k+1}$. Since $u \in \qbA{[k+1]}$ with $\es_i(u)=0$, the element $u$ can be written as a linear combination of monomials of $\rF_{2\ell}$, $\rF_{2\ell-1}$, \ldots,  $\rF_{\ell+2}$.  Thus the assertion follows from Lemma \ref{Lem: pbw comm}.

\noindent
\eqref{it: comm 2} Recall the anti $\bR$-automorphism $*: \qbA{} \buildrel \sim \over \longrightarrow \qbA{}$ in Section \ref{Sec: Bosonic ext}. Since $\wt(x) = \wt(*(x))$ for any homogeneous element $x$, the assertion follows by applying $*$ to~\eqref{it: comm 1}.
\end{proof}

We now define the adjoint operators on $\qbA{}$ as follows.

\begin{definition} \label{def: Adjoint operator}
For any $(i,k)\in \hI$, we define $\bR$-linear maps 
\begin{align*}
\Ep_{i,k}: \qbA{} \longrightarrow \qbA{} \quad \text{ and }\quad \Es_{i,k}: \qbA{} \longrightarrow \qbA{},
\end{align*}
by 
\begin{align*} 
\pair{ \Ep_{i,k} (x), y } = \pair{ x, f_{i,k}y } \quad \text{ and }\quad \pair{  \Es_{i,k} (x), y } = \pair{  x, yf_{i,k} }
\end{align*}
for any $x,y \in \qbA{}$. 
\end{definition}

Lemma  \ref{Lem: properties of (,)} and Corollary \ref{Cor: orthonormal_reduced_Aq} say that $\Ep_{i,k}$ and $\Es_{i,k}$ are well-defined $\bR$-linear maps.
It is easy to see that 
$$
\Es_{i,p} = * \circ  \Ep_{i,-p} \circ  *
$$
for any $(i,p)\in \hI$.

\begin{theorem} \label{Thm: E_ik, f_ik} 
Let $i\in I$, $k \in \Z$, and let 
$$
x \in \qbA{> k+1}, \quad y \in \qbA{[k+1]},\quad  z \in \qbA{[k]}, \quad  w \in \qbA{<k}
$$ 
be homogeneous elements. Then we have 
\bnum
\item \label{it: left f} $f_{i,k}(xyzw) = q_i^{ (-1)^k \langle  h_i, \wt(xy) \rangle-1 } \left(    (q_i^{-1}-q_i) x (\es_{i} (y))z w + q_i x y (f_{i,k} z)w  \right),$
\item \label{it: left E} $\Ep_{i,k}(xyzw) = q_i^{ (-1)^k \langle  h_{i}, \wt(xy) \rangle } \left( (q_i^{-1}-q_i) x y ( e_{i}'(z)) w + q_i x (yf_{i,k+1}) zw  \right),$
\item \label{it: right f} $(xyzw) f_{i,k+1} = q_i^{ (-1)^{k+1} \langle h_{i}, \wt(zw) \rangle-1 } \left( (q_i^{-1}-q_i) x y (e_{i}'(z))w + q_i x (y f_{i,k+1})z w \right),$
\item \label{it: right E} $\Es_{i,k+1}(xyzw) = q_i^{ (-1)^{k+1} \langle h_{i}, \wt(zw) \rangle } \left(  (q_i^{-1}-q_i) x (\es_{i} (y))z w + q_i x y (f_{i,k} z)w  \right).$
\ee
\end{theorem}
\begin{proof}
\eqref{it: left f} As $ \calU_q^-(\g) = \sum_{t=0}^\infty \calU_i f_i^t$, where $\calU_i$ is the subalgebra given by \eqref{eq: Ui},  we may assume that $y = y_0 f_{i, k+1}^m$ for some $y_0 \in \qbA{[k+1]}$ with $\es_i(y_0)=0$ and $m\in \Z_{\ge0}$. 
By \eqref{Eq: es e'i for f_i^m}, \eqref{Eq: defining rel},  Lemma \ref{Lem: fip fip+1} and Lemma \ref{Lem: commute for f_ik u}, we have 
\begin{align*}
f_{i,k} (xyzw) &= f_{i,k} (x (y_0 f_{i, k+1}^m) zw) \allowdisplaybreaks\\
&=  q_i^{ (-1)^k \langle h_{i}, \wt(xy_0) \rangle}  x  y_0  (f_{i,k}f_{i, k+1}^m) zw \allowdisplaybreaks\\
&= q_i^{ (-1)^k \langle h_{i}, \wt(xy_0) \rangle}  x  y_0  \left(
(1-q_i^{2m})f_{i, k+1}^{m-1} + q_i^{2m} f_{i, k+1}^m f_{i,k}
\right) zw \allowdisplaybreaks\\
&= q_i^{ (-1)^k \langle h_{i}, \wt(xy_0) \rangle}  \left(
(1-q_i^{2m}) x  (y_0  f_{i, k+1}^{m-1} ) zw  
 + q_i^{2m} x  y_0 f_{i, k+1}^m (f_{i,k}z)w
 \right) \allowdisplaybreaks\\
&= q_i^{ (-1)^k \langle h_{i}, \wt(xy_0) \rangle}  \left(
 q_i^{2m-1}(q_i^{-1}-q_i) x  \es_i(y)  zw  
+ q_i^{2m} x  y (f_{i,k}z)w
\right) \allowdisplaybreaks\\
&= q_i^{ (-1)^k \langle h_{i}, \wt(xy) \rangle - 1}  \left(
 (q_i^{-1}-q_i) x  \es_i(y)  zw  
+ q_i x  y (f_{i,k}z)w
\right),
\end{align*}
where the last equality follows from that 
$( \wt(f_{i,k}), \wt(xy) ) =  ( \wt(f_{i,k}), \wt(xy_0) ) - m(\al_i, \al_i) $.

\noindent
\eqref{it: left E} We choose orthogonal bases 
$$
B_1 \subset \qbA{>k+1}, \quad B_2 \subset \qbA{[k+1]},\quad B_3 \subset \qbA{[k]}, \quad B_4 \subset \qbA{<k}
$$ 
with respect to the bilinear form $\pair{\ , \ }$. We assume that $B_2$ and $B_3$ satisfy the following conditions:
\bna
\item Any element $y \in B_2$ has the form $y=y_0 f_{i,k+1}^m$ for some $m\in \Z_{\ge0}$ and $y_0 \in \qbA{[k+1]}$ with $ \es_i(y_0)=0$ and $\pair{y_0 f_{i,k+1}^m, y_0' f_{i,k+1}^{m'} } = \pair{ y_0, y_0'} \pair{ f_{i,k+1}^m, f_{i,k+1}^{m'} }  $.
\item Any element $z \in B_3$ has the form $z= f_{i,k}^t z_0$ for some $t\in \Z_{\ge0}$ and $z_0 \in \qbA{[k]}$ with $ e_i'(z_0)=0$ and $ \pair{f_{i,k+1}^t z_0 ,  f_{i,k+1}^{t'} z_0' } = \pair{ z_0, z_0'} \pair{ f_{i,k+1}^t, f_{i,k+1}^{t'} }  $.
\ee
Note that $B_2$ (resp.\ $B_3$) can be obtained by taking the PBW basis of $\qbA{[k+1]}$ (resp.\ $\qbA{[k]}$) associated with a reduced expression $(i_1, i_2, \ldots, i_\ell) \in R(w_\circ)$ with $i_1 = i$ (resp.\ $i_\ell = i^*$) (see Theorem \ref{Thm: orthonormal_reduced}). 
By Lemma \ref{Lem: properties of (,)}, the bases $B_1$, $B_2$, $B_3$ and $B_4$ are mutually orthogonal, i.e., the product $ B_1 \times B_2 \times B_3 \times B_4 $ becomes an orthogonal basis of $\qbA{}$.

We take $x,x' \in B_1$, $y,y' \in B_2$, $z,z' \in B_3$, and $w,w' \in B_4$ and write 
$$
y = y_0 f_{i,k+1}^m, \quad y' = y_0' f_{i,k+1}^{m'}, \quad z =  f_{i,k}^t z_0, \quad z' =  f_{i,k}^{t'} z_0' 
$$ 
for some $y_0, y_0' \in \qbA{[k+1]}$ with $ \es_i(y_0)= \es_i(y_0')=0$ and  
$z_0, z_0' \in \qbA{[k]}$ with $ e_i'(z_0)= e_i'(z_0')=0$. 
By~\eqref{it: left f}, we have 
\begin{equation} \label{Eq: E'ik}
\begin{aligned}
\pair{ \Ep_{i,k} & (xyzw),  x'y'z'w' } = \pair{ xyzw,  f_{i,k}( x'y'z'w')  } \\
&= q_i^{ (-1)^k \langle h_{i}, \wt(x'y')  \rangle -1 }
\pair{ xyzw, (q_i^{-1} -q_i )x' \es_i(y')z'w' + q_i x'y'(f_{i,k}z')w' } \\
&= q_i^{ (-1)^k \langle h_{i}, \wt(x'y')  \rangle -1 }  (q_i^{-1} -q_i ) \pair{x,x'} \pair{y, \es_i(y')} \pair{z,z'} \pair{w,w'} \\
&\quad  + q_i^{ (-1)^k \langle h_{i}, \wt(x'y') \rangle }   \pair{x,x'} \pair{y, y'} \pair{z, f_{i,k}z'} \pair{w,w'} \\
&= q_i^{ (-1)^k \langle h_{i}, \wt(x'y')  \rangle -1 }  (q_i^{-2m'+1} -q_i ) \pair{x,x'} \pair{y_0, y'_0} \pair{z,z'} \pair{w,w'} \pair{ f_{i,k+1}^m, f_{i,k+1}^{m'-1} } \\
&\quad  + q_i^{ (-1)^k \langle h_{i}, \wt(x'y') \rangle }   \pair{x,x'} \pair{y, y'} \pair{z_0, z_0'} \pair{w,w'} \pair{f_{i,k}^t, f_{i,k}^{t'+1}}. 
\end{aligned}
\end{equation}
Thus, since the product $ B_1 \times B_2 \times B_3 \times B_4 $ is an orthogonal basis of $\qbA{}$, we can write  
$$
\Ep_{i,k}(xyzw) = A x (y_0 f_{i,k+1}^{m+1}) zw + B xy(f_{i,k}^{t-1}z_0)w,
$$
for some $A,B \in \bR$. Applying $ \pair{ - ,\  x (y_0 f_{i,k+1}^{m+1}) zw } $ to the above identity, it follows from Theorem \ref{Thm: orthonormal_reduced} and \eqref{Eq: E'ik} that 
\begin{align*}
A &= q_i^{ (-1)^k \langle h_{i}, \wt(x (y_0 f_{i,k+1}^{m+1}) )  \rangle -1 }  (q_i^{-2m-1} -q_i )
\frac{   \pair{ f_{i,k+1}^m, f_{i,k+1}^{m} } }
{  \pair{ f_{i,k+1}^{m+1}, f_{i,k+1}^{m+1} }  } \allowdisplaybreaks\\
&= q_i^{ (-1)^k \langle h_{i}, \wt(x y )  \rangle +1 }  (q_i^{-2m-1} -q_i )
\frac{   q_i^m }
{  (q_i^{-1}-q_i) [m+1]_i  } \allowdisplaybreaks\\
&= q_i^{ (-1)^k \langle h_{i}, \wt(x y )  \rangle +1 } .
\end{align*}
In a similar manner, we have 
\begin{align*}
B &= q_i^{(-1)^k \langle h_{i}, \wt(xy) \rangle } \frac{\pair{ f_{i,k}^t, f_{i,k}^{t}}}{\pair{ f_{i,k}^{t-1}, f_{i,k}^{t-1}}} \allowdisplaybreaks \\
&= q_i^{ (-1)^k \langle h_{i}, \wt(xy) \rangle } q_i^{-t+1} (q_i^{-1} - q_i) [t]_i.
\end{align*}
Therefore, we obtain
\begin{align*}
\Ep_{i,k}(xyzw) &= q_i^{ (-1)^k \langle h_{i}, \wt(x y ) \rangle } \left( q_i x (y_0 f_{i,k+1}^{m+1}) zw + 
 q_i^{-t+1} (q_i^{-1} - q_i) [t]_i  xy(f_{i,k}^{t-1}z_0)w \right) \\
 &= q_i^{ (-1)^k \langle h_{i}, \wt(x y ) \rangle  } \left( q_i x (y f_{i,k+1}) zw + 
  (q_i^{-1} - q_i)  xy(e_i'(z))w \right),
\end{align*}
where the last equality follows from \eqref{Eq: es e'i for f_i^m}.

\smallskip

\noindent
\eqref{it: right f} and~\eqref{it: right E} follow by applying the anti-automorphism $*$ to~\eqref{it: left f} and~\eqref{it: left E}.
\end{proof}

\begin{corollary} \label{cor: Ep and f}
Let $i\in I$ and $k \in \Z$. 
For any homogeneous element $u \in \qbA{}$, we have
\begin{align*}
\Ep_{i,k} (u) &= q_i^{ (-1)^k \langle h_{i}, \wt(u)\rangle + 1 } u  f_{i, k+1}, \\
\Es_{i,k} (u) &= q_i^{ (-1)^k \langle h_{i}, \wt(u)\rangle + 1} f_{i, k-1} u.
\end{align*}
\end{corollary}
\begin{proof}
Let $x \in \qbA{>k+1}$, $y \in \qbA{[k+1]}$, $z \in \qbA{[k]}$, and $w \in \qbA{<k}$ be homogeneous elements. We set 
$$ X \seteq (q_i^{-1}-q_i) x y ( e_{i}'(z)) w + q_i x (yf_{i,k+1}) zw. $$ 
By Theorem \ref{Thm: E_ik, f_ik}, we have 
\begin{align*}
\Ep_{i,k} (xyzw) &=  q_i^{ (-1)^k \langle  h_{i}, \wt(xy) \rangle } X \\
& = q_i^{ (-1)^k \langle  h_{i}, \wt(xyzw) \rangle + 1 } q_i^{ (-1)^{k+1} \langle  h_{i}, \wt(zw) \rangle - 1 }  X \\
& = q_i^{ (-1)^k \langle  h_{i}, \wt(xyzw) \rangle + 1 } (xyzw) f_{i, k+1},
\end{align*} 
which gives the first identity. 
In a similar manner, one can prove the second identity. 
\end{proof}

\begin{corollary} \label{cor: q to qinverse}
We have the following relations in $\End(\hcalA):$
\bnum
\item $\displaystyle\sum^{1-c_{i,j}}_{k=0} (-1)^k \left[\begin{matrix} 1-c_{i,j} \\ k \end{matrix} \right]_i  E_{i,p}^{\prime 1-c_{i,j}-k} \Ep_{j,p} E_{i,p}^{\prime k} = 0  $ for $i \ne j$.
\item $\Ep_{i,m}\Ep_{j,p} =q_i^{(-1)^{p-m}c_{i,j}}\Ep_{j,p}\Ep_{i,m}$ for $p>m+1$. 
\item $\Ep_{i,k}\Ep_{j,k+1} =q_i^{-c_{i,j}}\Ep_{i,k+1}\Ep_{i,k} + \delta_{i,j}(1-q_i^{-2})$. 
\item $\Ep_{i,k}\Es_{j,l} = \Es_{j,l}\Ep_{i,k}$.
\ee
\end{corollary}

\begin{remark}
The first three relations in Corollary~\ref{cor: q to qinverse} can be obtained from 
the defining relations in Definition~\ref{Def: extended qg} by replacing $q^{1/2}$ with $q^{-1/2}$. These coincide with the defining relations in \cite[Thoerem 7.3]{HL15}. The operators $\{ \Es_{i,k} \}_{(i,k)\in \hI}$ satisfies the same relations in Definition~\ref{Def: extended qg}. 
\end{remark}

The following tells us that the operators $E_{i,p}$ and $\Es_{i,p}$ have analogues of the $q$-derivation properties \eqref{Eq: ei ei*} of $e'_i$ and $\es_i$.
\begin{corollary} \label{cor: prime operation}
Let $i\in I$ and $k \in \Z$. 
\bnum
\item \label{it: prime 1} For any homogeneous elements $ u \in \qbA{} $, $v \in \qbA{[ k]}$ and $w \in \qbA{<k}$, we have 
\begin{align*}
\Ep_{i,k}(uvw) = \Ep_{i,k}(u) vw + q_i^{ (-1)^k \langle h_{i}, \wt(u) \rangle } (q_i^{-1}-q_i) ue_{i}' (v)w.
\end{align*}  
\item \label{it: prime 2} For any homogeneous elements $ u \in \qbA{>k} $, $v \in \qbA{[k]}$ and $w \in \qbA{}$, we have 
\begin{align*}
\Es_{i,k}(uvw) =  uv\Es_{i,k}(w) + q_i^{ (-1)^k \langle h_{i}, \wt(w)\rangle } (q_i^{-1}-q_i) u\es_{i} (v)w.
\end{align*}
\ee
\end{corollary}
\begin{proof}
\eqref{it: prime 1} By \eqref{Eq: defining rel}, Theorem \ref{Thm: E_ik, f_ik} and Corollary \ref{cor: Ep and f}, we have 
\begin{align*}
\Ep_{i,k}(uvw) &= q_i^{ (-1)^k \langle h_{i}, \wt(uvw)\rangle + 1 } (uvw)  f_{i, k+1} \allowdisplaybreaks\\
&= q_i^{ (-1)^k \langle h_{i}, \wt(uv )\rangle + 1 } u(v f_{i, k+1} ) w \allowdisplaybreaks\\ 
&= q_i^{ (-1)^k \langle h_{i}, \wt(u)\rangle  } \left(  q_i (u f_{i, k+1}) v w + (q_i^{-1}-q_i) u e'_i(v)w  \right) \allowdisplaybreaks\\
&=    \Ep_{i,k} (u) v w +  q_i^{ (-1)^k \langle h_{i}, \wt(u)\rangle  } (q_i^{-1}-q_i) u e'_i(v)w.  
\end{align*}

\eqref{it: prime 2} It can be proved in the same manner as above.
\end{proof}

 \vskip 2em

\section{PBW bases: arbitrary sequences} \label{Section: general PBW}

In this section, we construct and investigate PBW bases for arbitrary sequences. 
This allows us to construct a subalgebra $\qbA{}(\ttb)$ of $\qbA{}$ associated with an element $\ttb$ of the braid group $\ttB_\cmC$. The subalgebra $\qbA{}(\ttb)$ can be understood as a \emph{braid-analogue} of the unipotent quantum coordinate ring $A_q(\n(w))$ associated with an element $w$ of the Weyl group $\weyl_\cmC$.
We shall prove that the PBW monomials associated with any sequence of
$\ttb$ form an orthogonal basis of $\qbA{}(\ttb)$ with respect to the pairing $\pair{ \ , \ }$.

\subsection{Garside normal form}
In this subsection, we briefly recall the basic results of a presentation of $\ttb $ in a Braid group referred to as \emph{Garside normal form} for our purpose (see \cite{Gar96} and see also \cite{MPfactor} and references therein).

We simply write $\ttB = \ttB_\cmC $ and $\weyl = \weyl_\cmC $.
Recall the natural projection $ \pi: \ttB^{+} \rightarrow \weyl$ given in \eqref{Eq: b->W} and the element $\gar \in \ttB^+$  such that $\ell(\gar) = \ell(w_\circ)$ and $ \pi( \gar) = w_\circ$. Note that $\gar^2$  belongs to  the center of $\ttB$. 
For $\ttx,\tty \in \ttB^+$, we write $\ttx \le \tty$ if there exists $\ttz \in \ttB^+$ such that $\ttx\ttz=\tty$. We say $\ttx$ a \emph{prefix} of $\tty$, and a prefix $\ttx$ of $\gar$ an \emph{permutation braid}. 
Note that the partial order $\le$  is invariant under left multiplication and the image of permutation braids under $\pi$ coincides with $\weyl$.

\begin{proposition} [{\cite{Gar96}}]\label{prop: gcd}
For any elements $\ttx,\tty \in \ttB^+$, there exists a unique element $\ttd$ such that $\ttd \le \ttx$, $\ttd \le \tty$ and that $\ttd' \le \ttd$ for every common prefix $\ttd'$
of $\ttx$ and $\tty$. 
\end{proposition}

We call $\ttd$ in Proposition~\ref{prop: gcd} the \emph{greatest common divisor} of $\ttx$ and $\tty$, and denote it by $\ttx \wedge \tty$.   
A product $\ttx\tty$ of permutation braids $\ttx$ and $\tty$  is \emph{left-weighted} if $\ttx\tty \wedge \gar=\ttx$.

\begin{theorem} [{\cite{WP,EM94}}] \label{thm: Garside} 
$($Garside left normal form$)$ \
Each element $\ttb \in \ttB$ can be presented by a product of permutation braids of the form
$$ 
\gar^r \ttx_1 \cdots \ttx_k,
$$   
where $r \in \Z$, $k \in \Z_{\ge 0}$, $1 < \ttx_i < \gar$ and $\ttx_i\ttx_{i+1}$ is left-weighted for $1 \le i <k$. 
\end{theorem}

\begin{corollary} 
\label{Cor: braid} 
For any $\ttx \in \ttB^+$, there exists $\tty\in \ttB^+$ and $m\in \Z_{\ge 0}$ such that $ \ttx\tty = \gar^m$. 
\end{corollary}
\begin{proof}
By Theorem~\ref{thm: Garside}, $\ttx$ can be presented by $\gar^r \ttx_1 \cdots \ttx_k$. Take $\tty_i \in \ttB^+$
such that $\ttx_i\tty_i =\gar$ for $1 \le i \le k$ and set
$$
\tty \seteq (\tty_k\gar)(\tty_{k-1}\gar) \cdots (\tty_1\gar) \in \ttB^+. 
$$
Then the assertion follows from the fact that $\gar^2$ belongs to the center of $\ttB$. 
\end{proof}

\subsection{Subalgebras $\qbA{}(\ttb)$} 
In this subsection we define a subalgebra of $\qbA{}(\ttb)$ of $\qbA{}$ associated with an element $ \ttb \in \ttB^+$.

\begin{definition} \label{Def: Aq(P;r,s)} \
\bnum 
\item A pair $\pd = (\ii, \xi)$ is called a PBW datum of $\ttb \in \bg^+$ if $\ii \in \Seq(\ttb)$ and $\xi \in \Z$.

\item Let $\pd = (\ii, \xi)$ be a PBW datum of $\ttb \in \bg^+$ and set $N\seteq \ell(\ttb)$. For any interval $\ttJ \subseteq [1,N]$, we define $ \qbAu{\ttJ,\pd} $ to be the $\bR$-subalgebra of $\qbA{}$ generated by the PBW root vectors $\rF^{\pd}_k$ for $  k \in \ttJ$.
When $\xi = 0$ and $\ttJ=[1,N]$, we simple write 
$$
\qbA{}(\ttb)\seteq \qbA{}(\ii) = \qbAu{[1,N],\pd}.
$$  
\ee
\end{definition}
We will deal with the well-definedness for $\qbA{}(\ttb)$ in Corollary \ref{Cor: well-def} below.

\smallskip

Under the isomorphism $\Um \simeq \qbA{[k]}$ for each $k \in \Z$, it follows from \eqref{eq: PBW basis} and Proposition \ref{Prop: T and S} that  
\begin{equation} \label{Eq: braid qA}
\begin{aligned}
\qbA{}(\hspace{-1.3ex}\underbrace{i,j,\cdots}_{m(i,j)\text{-factors}}\hspace{-1.3ex}) &
= \qbA{}(\hspace{-1.5ex}\underbrace{j,i\cdots}_{m(i,j)\text{-factors}}\hspace{-1.3ex})  \quad \text{ for any $i \ne j \in I$},  
\end{aligned}
\end{equation}
where $m(i,j)$ is given in \eqref{Eq: m(i,j)}.

From Definition \ref{Def: Aq(P;r,s)}, we have the following.

\begin{lemma} \label{Lem: qA(pd, rs)}
Let $\pd = (\ii, \xi)$ be a PBW datum of $\ttb \in \bg^+$, and write $\ii = (i_k)_{k\in [1,N]}$. 
\bnum
\item \label{it: qA1} Let $r,s \in [1,N]$ with $r \le s$.
We set $\ii' = (i_r, i_{r+1}, \ldots, i_N )$, and define $\pd' \seteq (\ii', \xi)$. Then we have 
$$
\qbAu{[r, s],\pd}  = \bT_{i_{1}} \bT_{i_{2}} \cdots \bT_{i_{r-1}} \left( \qbAu{[1,s-r+1],\pd'} \right). 
$$
\item Let $k\in \Z$ and set $\pd'' \seteq (\ii, \xi+k) $. Then we have 
$$
\qbAu{[r, s],\pd''} = \dD^k \left( \qbAu{[r, s],\pd}  \right).
$$
\ee
\end{lemma}

\subsection{PBW bases for $\qbA{}(\ttb)$}
In this subsection, we shall construct the PBW bases of $\qbA{}(\ttb)$ for arbitrary element $\ttb \in \bg^+$. 

We first focus on the case where $\ttb = \gar^m$. 
Let $\pd = (\ii, \xi)$ be a PBW datum of $\gar^m \in \bg^+$. We set $N\seteq \ell(\gar^m)  = m \ell$ and write $\ii = (i_k)_{k\in [1,N]} \in \Seq(\gar^m)$.

\begin{assumption} \label{Assu: pbw} \
\bna
\item \label{it: a1} The set $\sfP_\pd \seteq \{ \rF_{\pd}(\bsa) \mid \bsa \in \Z_{\ge 0}^{[1,N]} \}$ is a $\bR$-linear orthogonal basis of $\qbAu{[1, N],\pd}$ with respect to the bilinear form $\pair{\ , \ }$.   
\item \label{it: a2} For any $\bst, \bst' \in \Z_{\ge0}^{[1,N]}$, we have 
$$
\pair{\rF_{\pd}(\bst), \rF_{\pd}(\bst')} = \prod_{k=1}^N \delta_{t_k, t_k'} q_{i_k}^{-t_{k}(t_{k}-1)/2} (q_{i_k}^{-1}-q_{i_k})^{t_k}  [t_{k}]_{i_k}!.
$$
\item \label{it: a3} 
For any $k, t \in [1,N] $ with $ k < t$, we have 
$$
\rF^{\pd}_k \rF^{\pd}_t - q^{ - ( \wt(\rF^{\pd}_k), \wt (\rF^{\pd}_t) ) } \rF^{\pd}_t \rF^{\pd}_k = \sum_{\bsc \in \Z_{\ge 0}^{\oplus(k, t) }} A_\bsc  \rF_{\pd}(\bsc) 
\quad \text{ for some $A_\bsc \in \bR$.}
$$
\ee
\end{assumption} 

\begin{lemma} \label{Lem: PBW datum-orthogonal}
Suppose that a PBW datum $\pd = (\ii, \xi)$ of $\gar^m$ satisfies Assumption \ref{Assu: pbw}. Let $r,t,s \in [1,N]$ with $r  \le  t<s$.
\bnum
\item As a $\bR$-vector space, we have  
$$
\qbAu{[r,s],\pd}  \simeq \qbAu{(t,s],\pd}  \otimes_{\bR} \qbAu{[r,t],\pd} .  
$$
\item Let $B_1$ and $B_2$ be $\bR$-linear orthogonal bases of $\qbAu{(t,s],\pd}$ and $\qbAu{[r,t],\pd}$ respectively. Then the set $ B_1 \times B_{2} $ is a $\bR$-linear orthogonal basis of $\qbAu{[r,s],\pd}$.
\item For any $x,x' \in \qbAu{(t,s],\pd}$ and $y,y' \in \qbAu{[r,t],\pd}$, we have 
$$ \pair{xy, x'y'} = \pair{x,x'}\pair{y,y'}. $$
\ee
\end{lemma}
\begin{proof}
It follows from Assumption \ref{Assu: pbw} directly. 
\end{proof}

\begin{lemma} \label{Lem: initial}
There is a sequence $\ii \in \Seq(\gar^m)$ such that, for any $\xi\in \Z$, the PBW datum $\pd = (\ii, \xi)$ satisfies  {\rm Assumption \ref{Assu: pbw}}. 
\end{lemma}
\begin{proof}
We choose a locally reduced sequence $(i_k)_{k\in \Z}   \in I^\Z$ satisfying \eqref{Eq: longest case} and set 
$\bfi \seteq (i_k)_{k\in [1,m\ell]}$. It is obvious that $\bfi \in \Seq(\gar^m)$. Then the assertion follows from Theorem \ref{Thm: orthonormal_reduced} and Lemma \ref{Lem: pbw comm}.
\end{proof}

For $i \ne j\in I$, we set 
$$
\eta(i,j) \seteq m(i,j)-1 
$$
The following is the main theorem of this subsection.
\begin{theorem} \label{Thm: key thm1}
Suppose that a PBW datum $\pd = (\ii, \xi)$ of $\gar^m$ satisfies {\rm Assumption \ref{Assu: pbw}}. Let $\jj$ be the sequence obtained from $\ii$ by applying a single braid move. 
Then the PBW datum $\pd' = (\jj, \xi)$ also satisfies {\rm Assumption \ref{Assu: pbw}}.
\end{theorem}
\begin{proof}
Let $\ii = (i_k)_{k\in [1,N]}$ and let $\jj = (j_k)_{k\in [1,N]}$.
Suppose that $\jj$ is obtained from $\ii$ by applying a single braid move at the $c$-th position, i.e., 
$ i_k = j_k $ for either $ k < c $ or $ c+\eta < k$ and $(i_c, i_{c+1}, \ldots, i_{c+\eta})$ and $(j_c, j_{c+1}, \ldots, j_{c+\eta})$ are related by an $m(i_c,i_{c+1})$-braid move, where $\eta\seteq \eta(i_c,i_{c+1})$. 
Note that, for either $k <c$ or $k > c+\eta$, we have 
$$
\rF^{\pd}_k = \rF^{\pd'}_k
$$
by Proposition~\ref{Prop: Ti braid}.
It follows from \eqref{Eq: braid qA} and 
Lemma~\ref{Lem: qA(pd, rs)} that 
\begin{align*}
\qbAu{[1,c),\pd}  &= \qbAu{[1,c),\pd'}, \quad
\qbAu{[c,c+\eta],\pd} = \qbAu{[c,c+\eta],\pd'} \qtq 
\qbAu{(c+\eta,N],\pd} = \qbAu{(c+\eta, N],\pd'}.
\end{align*}
By Assumption \ref{Assu: pbw} and Lemma \ref{Lem: PBW datum-orthogonal}, the PBW datum $\pd' = (\jj, \xi)$ satisfies Assumption \ref{Assu: pbw}~\eqref{it: a1} and~\eqref{it: a2}.

We shall prove that $\pd' = (\jj, \xi)$ satisfies Assumption \ref{Assu: pbw}~\eqref{it: a3}.
Let $s,t \in [1,N]$. 

\noindent
$\bullet$ If $s,t \in [1,c) \cup (c+\eta,N]$, then it follows from Lemma \ref{Lem: qA(pd, rs)} and Assumption \ref{Assu: pbw}. 

\noindent
$\bullet$ If $s,t \in [c, c+\eta]$, then it follows from Lemma~\ref{Lem: pbw comm} and Lemma~\ref{Lem: qA(pd, rs)}.

We first assume that $ s \in [c, c+\eta] $ and $ t \in (c+\eta, N] $.
Since $ \rF^{\pd'}_s \in \qbAu{[c,c+\eta],\pd}$, by \eqref{Eq: braid qA}, Lemma \ref{Lem: qA(pd, rs)} and Assumption \ref{Assu: pbw},
one can write 
\begin{align} \label{eq: F_P's F_Pt}
\rF^{\pd'}_s \rF^{\pd'}_t = \sum_{\bsd } A_\bsd \rF_{\pd'}(\bsd) \quad \text{ for some $A_\bsd \in \bR$},
\end{align}
where $\bsd = (d_c,d_{c+1},\ldots,d_t)$ runs over all elements in $ \Z_{\ge0}^{\oplus[c, t]} $. 
Let $\{ \bse_k \}_{c \le k \le t}$ be a natural basis of $ \Z_{\ge0}^{\oplus[c, t]} $. 

\smallskip

We define 
\begin{align*}
\rG_k &\seteq \bT_{j_s} \bT_{j_{s+1}} \cdots \bT_{j_{k-1}} (f_{j_k,\xi}) \ \ \qquad \text{ for $k > s$,}\\
\rH_k &\seteq \bT^{-1}_{j_{s-1}} \bT^{-1}_{j_{s-2}} \cdots \bT^{-1}_{j_{k+1}} (f_{j_k,\xi-1})\quad \text{ for $k < s$,}
\end{align*}
and set 
\begin{align*}
\rG^{\bsd^+} &\seteq \rG_t^{d_t} \rG_{t-1}^{d_{t-1}} \cdots \rG_{s+1}^{d_{s+1}}     \qquad \text{ for $\bsd^+ \in  \Z_{\ge 0}^{\oplus (s,t]}$,}\\
\rH^{\bsd^-} &\seteq  \rH_{s-1}^{d_{s-1}} \rH_{s-2}^{d_{s-2}} \cdots \rH_{c}^{d_{c}}   \qquad \text{ for $\bsd^- \in  \Z_{\ge 0}^{\oplus [c,s)}$.}
\end{align*}
Note that Lemma \ref{Lem: codomain of T_i} says that 
\begin{align*} 
\rG^{\bsd^+} \in (\qbA{\ge \xi})_{j_s} \quad  \text{and} \quad \rH^{\bsd^-} \in \qbA{<\xi}.
\end{align*}
Applying $\bT_{j_{s-1}}^{-1} \bT_{j_{s-2}}^{-1} \cdots \bT_{j_{1}}^{-1} $ to \eqref{eq: F_P's F_Pt}, we have 
\begin{align*}
f_{j_s, \xi} \rG_{t} = \sum_{\bsd^+, \bsd^-, d_s} A_\bfd  \rG^{\bsd^+}  f_{j_s, \xi}^{d_s} \rH^{\bsd^-},
\end{align*}
where $\bsd = \bsd^- * (d_s) * \bsd^+ \in \Z_{\ge0}^{\oplus [c,t]} $ and $*$ is the concatenation.
By~\eqref{Eq: es e'i for f_i^m} and Corollary \ref{cor: prime operation}, we have 
\begin{align*}
\Bpair{ f_{j_s, \xi} \rG_{t}, \rG^{\bsd^+}  f_{j_s, \xi}^{d_s} \rH^{\bsd^-} } & = \Bpair{ \rG_{t}, \Ep_{j_s,\xi} \left( \rG^{\bsd^+}  f_{j_s, \xi}^{d_s} \rH^{\bsd^-} \right) } \\
&= \Bpair{ \rG_{t}, \Ep_{j_s,\xi} \left( \rG^{\bsd^+} \right)  f_{j_s, \xi}^{d_s} \rH^{\bsd^-}  }  +  B \Bpair{ \rG_{t},  \rG^{\bsd^+}  f_{j_s, \xi}^{d_s-1} \rH^{\bsd^-}  }
\end{align*}
where $ B = q_{j_s}^{ (-1)^{\xi} \langle h_{j_s}, \wt(\rG^{\bsd^+}) \rangle -d_s+1 } [d_s]_{j_s} (q_{j_s}^{-1}-q_{j_s})  $ if $d_s>0$ and $B=0$ otherwise.
By the orthogonality, the value of the above bilinear form is $0$ in one of the following cases:
\bna
\item $\bsd^- \ne (0,0,\ldots, 0) $,
\item $d_s > 1$.
\item $\bsd^- = (0,0,\ldots, 0) $, $d_s = 1$ and $ \rG^{\bsd^+} \ne \rG_{t} $.
\ee
Thus $A_\bsd = 0$ in the above cases.
In particular, when 
$\bsd^- = (0,0,\ldots, 0)$, $d_s = 1$ and $\rG^{\bsd^+} = \rG_{t}$,
Theorem \ref{Thm: orthonormal_reduced} and Corollary~\ref{cor: prime operation} tell us that
$$
\pair{ f_{j_s, \xi} \rG_{t}, \rG_{t}  f_{j_s, \xi} } =
\pair{  \rG_{t},  \Ep_{j_s, \xi} (\rG_{t}  f_{j_s, \xi}) } =
q_{j_s}^{ (-1)^\xi \langle h_{j_s}, \wt(\rG_t) \rangle  } (q_{j_s}^{-1}-q_{j_s})(q_{j_t}^{-1}-q_{j_t}),
$$
and hence 
$$
A_\bsd = \frac{ \pair{ f_{j_s, \xi} \rG_{t}, \rG_{t}  f_{j_s, \xi} } }{ \pair{ \rG_{t}  f_{j_s, \xi}, \rG_{t}  f_{j_s, \xi} } } = q_{j_s}^{ (-1)^\xi \langle h_{j_s}, \wt(\rG_t) \rangle  }.
$$
Since the bilinear form is invariant under the actions $\bT_i$, 
we have 
\begin{align} \label{eq: F_P's F_Pt 3}
\rF^{\pd'}_s \rF^{\pd'}_t =  q^{ - ( \wt(\rF^{\pd'}_s), \wt(\rF^{\pd'}_t) )  } \rF^{\pd'}_t \rF^{\pd'}_s +  
\sum_{\bsd } A_\bsd \rF_{\pd'}(\bsd),
\end{align}
where $\bsd = (d_k)$ runs over  $\Z_{\ge0}^{\oplus(s, t]}$.

 Applying $\bT_{j_{t-1}}^{-1} \bT_{j_{t-2}}^{-1} \cdots \bT_{j_{1}}^{-1} $ to \eqref{eq: F_P's F_Pt 3}, we have 
\begin{align*}
\sfR_s f_{j_t, \xi}  = 
q^{ - ( \wt(\rF^{\pd'}_s), \wt(\rF^{\pd'}_t) )  } f_{j_t, \xi} \sfR_s  +
\sum_{\bsd = \bsd' *d_t  \in \Z_{\ge0}^{\oplus(s, t]}} A_\bsd  f_{j_t, \xi}^{d_t}  \sfR^{\bfd'},
\end{align*}
where $ \sfR_k=  \bT^{-1}_{j_{t-1}} \bT^{-1}_{j_{t-2}} \cdots \bT^{-1}_{j_{k+1}} (f_{j_k,\xi-1}) \in \qbA{<\xi}$ for $s\le  k <t$ and 
$$\sfR^{\bsd'} \seteq \sfR_{t-1}^{d_{t-1}}\cdots  \sfR_{s+2}^{d_{s+2}}\sfR_{s+1}^{d_{s+1}}  \in \qbA{<\xi} \quad \text{ for } \bsd'=(d_{s+1},\ldots,d_{t-1}) \in \Z^{\oplus (s,t)}.$$ 
By Corollary \ref{cor: prime operation}, for any $\bsd'=(d_{s+1},\ldots,d_{t-1}) \in \Z^{\oplus (s,t)}$, we have
\begin{align*}
\Bpair{\sfR_s f_{j_t, \xi}, f_{j_t, \xi}^{d_t}  \sfR^{\bfd'} } &= \Bpair{\sfR_s, \Es_{j_t,\xi} \left(f_{j_t, \xi}^{d_t}  \sfR^{\bfd'} \right) }     \\
& =  \Bpair{\sfR_s, f_{j_t, \xi}^{d_t}  \Es_{j_t,\xi} \left(\sfR^{\bfd'} \right) } + B'   \Bpair{\sfR_s, f_{j_t, \xi}^{d_t-1}   \sfR^{\bfd'}  } \ \ \text{ for some } B' \in \bfk,
\end{align*}
which tells that $A_\bsd$ vanishes unless $d_t=0$ by the orthogonality. 
Thus, with~\eqref{eq: F_P's F_Pt 3}, we can conclude that
$$\rF^{\pd'}_s \rF^{\pd'}_t =  q^{ - ( \wt(\rF^{\pd'}_s), \wt(\rF^{\pd'}_t) )  } \rF^{\pd'}_t \rF^{\pd'}_s +  
\sum_{\bsd } A_\bsd \rF_{\pd'}(\bsd),$$ 
where $\bsd = (d_k)$ runs over $\Z_{\ge 0}^{\oplus(s, t)}$. 
Therefore, $\pd' = (\jj, \xi)$ satisfies Assumption \ref{Assu: pbw}~\eqref{it: a3}.  

\smallskip 

The remaining case  for $ s \in [1,c) $ and  $ t \in [c, c+\eta] $  can be proved in the same manner as above.
\end{proof}

\begin{corollary} \label{Cor: key}
Every PBW datum $\pd = (\ii, \xi)$ of $\gar^m$ satisfies {\rm Assumption \ref{Assu: pbw}}.
\end{corollary}
\begin{proof}
Since any two elements in $\Seq(\ttb)$ of $\ttb \in \bg^+$ are related by braid relations, it follows from Theorem \ref{Lem: initial} and Theorem \ref{Thm: key thm1}.
\end{proof}

We now construct the PBW bases for $\qbA{}(\ttb)$ for arbitrary element $\ttb \in \bg^+$, which is the main theorem of the paper.

\begin{theorem} \label{Thm: main}
Let $\pd = (\ii, \xi)$ be a PBW datum of $\ttb \in \bg^+ $ and let $L = \ell(\ttb)$. 
For any $r,s \in [1,L]$ with $r \le s$, we have the following.
\bnum
\item The set 
$$ \sfP_{\pd;[r,s]} \seteq \left\{ \rF_{\pd}(\bsa) = \sfF_{s}^{a_s}\sfF_{s-1}^{a_{s-1}} \cdots \sfF_{r}^{a_r} \mid \bsa \in \Z_{\ge 0}^{\oplus [r,s]} \right\}$$ 
is a $\bR$-linear orthogonal basis of $\qbAu{[r, s],\pd}$ with respect to the bilinear form $\pair{ \ , \ }$.   
\item For any $\bst, \bst' \in \Z_{\ge0}^{\oplus [r,s]}$, we have 
$$
\pair{\rF_{\pd}(\bst), \rF_{\pd}(\bst')} = \prod_{k=1}^N \delta_{t_k, t_k'} q_{i_k}^{-t_{k}(t_{k}-1)/2} (q_{i_k}^{-1}-q_{i_k})^{t_k}  [t_{k}]_{i_k}!.
$$
\item  
For any $k, t \in [r,s] $ with $ k < t$, we have 
$$
\rF^{\pd}_k \rF^{\pd}_t - q^{ - ( \wt(\rF^{\pd}_k), \wt (\rF^{\pd}_t) ) } \rF^{\pd}_t \rF^{\pd}_k = \sum_{\bsc \in \Z_{\ge 0}^{\oplus(k, t) }} A_\bsc  \rF_{\pd}(\bsc) 
\quad \text{ for some $A_\bsc \in \bR$.}
$$
\ee
\end{theorem}
\begin{proof}
Thanks to Lemma \ref{Lem: qA(pd, rs)}~\eqref{it: qA1}, we may assume that $r=1$.
By Corollary \ref{Cor: braid}, there exists $\ttb'\in \bg_\cmC^+$ and $m\in \Z_{\ge0}$ such that $\ttb\ttb' = \gar^m$. 
We set $\jj \seteq \ii* \ii'$ for some $\ii' \in \Seq(\ttb')$. Then $\pd' \seteq (\jj, \xi)$ is a PBW datum of $\gar^m$, and the assertion follows from Lemma \ref{Lem: PBW datum-orthogonal} and Corollary \ref{Cor: key}.
\end{proof}

\begin{corollary} \label{Cor: well-def}
Let $\ttb \in \bg^+$. For any $\ii, \jj \in \Seq(\ttb)$, we have $ \qbA{}(\ii) = \qbA{}(\jj)$.  
Thus, the algebra $\qbA{}(\ttb)$ does not depend on the choice of $\ii \in \Seq(\ttb)$.    
\end{corollary}

\begin{example} 
We keep all notations appeared in Example \ref{Ex: example1}.
Let $\bg_{A_2}$ be the braid group of type $A_2$ and let $\ttb = r_1r_2r_2r_1 \in \bg_{A_2}^+$. We set 
$$
\jj\seteq (1,2,2,1) \in \Seq(\ttb).
$$

\bnum

\item We consider the PBW datum $ \pdh = (\jj, 0)$. Then the PBW vectors are given as follows:  
\begin{align*}
\rF^{\pdh}_1 &= f_{1,0}, \\
\rF^{\pdh}_2 &= \bT_1(f_{2,0}) =  \cC^{-1} [f_{1,0}, f_{2,0}]_q \\
\rF^{\pdh}_3 &= \bT_1\bT_2(f_{2,0}) = \bT_1(f_{2,1}) =  \cC^{-1} [f_{1,1}, f_{2,1}]_q, \\
\rF^{\pdh}_4 &= \bT_1\bT_2\bT_2(f_{1,0}) = \cC^{-1}\bT_1\bT_2([f_{2,0}, f_{1,0}]_q) = \cC^{-1}([\bT_1\bT_2(f_{2,0}), \bT_1\bT_2(f_{1,0})]_q)  \\
&= \cC^{-2} [[f_{1,1}, f_{2,1}]_q, f_{2,0}]_q.
\end{align*}
By Theorem \ref{Thm: main}, the PBW vectors $\rF^{\pdh}_1$, $\rF^{\pdh}_2$, $\rF^{\pdh}_3$ and $\rF^{\pdh}_4$ generate the algebra $\qbA{}(\ttb)$ and the PBW monomials $\rF_{\pdh}(\bsa)$ form an orthogonal basis of $\qbA{}(\ttb)$.

\item For any $ k, t \in [1,4]$ with $k < t$,  we have  $[\rF^{\pdh}_k, \rF^{\pdh}_t  ]_q=0$ except the following two cases:
\bna
\item By Definition \ref{Def: extended qg}, we have 
$$
[\rF^{\pdh}_2, \rF^{\pdh}_3 ]_q = \bT_1 ( [f_{2,0}, f_{2,1} ]_q ) 
= \bT_1(1-q^2) = 1-q^2.
$$
\item Since $[ f_{1,0}, f_{2,1}]_q=0 $, we have 
$$
[f_{1,0}, [f_{1,1}, f_{2,1}]_q  ]_q = [[f_{1,0}, f_{1,1}]_q, f_{2,1}  ]_q = [1-q^2, f_{2,1}  ]_q = 0.
$$
We thus have 
\begin{align*}
[ \rF^{\pdh}_1,  \rF^{\pdh}_4]_q &= 
 \cC^{-2}[f_{1,0}, [[f_{1,1}, f_{2,1}]_q, f_{2,0}]_q ]_q \\
 &= q^{ (\al_1, \al_1 + \al_2)} \cC^{-2} [f_{1,1},f_{2,1} ]_q [f_{1,0}, f_{2,0}]_q  - q^{(\al_1+\al_2, \al_2)} \cC^{-2} [f_{1,0}, f_{2,0}]_q [f_{1,1}, f_{2,1}]_q \\ 
& = q \rF^{\pdh}_3 \rF^{\pdh}_2  - q \rF^{\pdh}_2 \rF^{\pdh}_3 \\
&= q^2(q^{-1}-q) \rF^{\pdh}_3 \rF^{\pdh}_2 - q^2(q^{-1}-q).
\end{align*}
\ee
\item Let $\ttb' = r_2r_2 \in \bg^+_{A_2}$ and set $\jj'\seteq (2,2)$. Then 
$$
\ttb \ttb' = r_1r_2r_2r_1 r_2r_2 = r_1r_2 r_1r_2 r_1 r_2 = \gar^2. 
$$
Note that  $ \jj * \jj' \in \Seq(\gar^2) $ and $ \jj * \jj'$ is obtained from $\ii$ by applying 3-braid move to the third position. Let $ \pdh'\seteq (\jj * \jj', 0 )$. Then we have the corresponding PBW vectors
\begin{align*}
\rF^{\pdh'}_1 &= f_{1,0}, \qquad \rF^{\pdh'}_2 = \cC^{-1} [f_{1,0}, f_{2,0}]_q, \qquad \rF^{\pdh'}_3 = \cC^{-1} [f_{1,1}, f_{2,1}]_q, \\
\rF^{\pdh'}_4 &= \cC^{-2} [[f_{1,1}, f_{2,1}]_q, f_{2,0}]_q, \qquad \rF^{\pdh'}_5 = f_{2,0}, \qquad \rF^{\pdh'}_6 = f_{2,1},
\end{align*}	
which tells us that
\bna
\item $\rF^{\pdh}_k  = \rF^{\pdh'}_k$ for $k=1,2,3,4$,
\item $\rF^{\pdh'}_k = \rF^{\pd}_k$ for $k=1,2,6$,
\item $\rF^{\pdh'}_3 = \rF^{\pd}_5$ and $\rF^{\pdh'}_5 = \rF^{\pd}_3$.
\ee
\ee

\end{example}

\bibliographystyle{amsplain}
\bibliography{ref}{}

\end{document}